\documentclass[12pt,letterpaper]{article}
\usepackage{color}
\usepackage{amssymb,amsmath,amsthm}
\usepackage{subcaption}
\usepackage{multirow}
\usepackage{graphicx}
\usepackage{booktabs}
\usepackage{sectsty}
\usepackage[numbers,square]{natbib}
\sectionfont{\normalsize}
\subsectionfont{\normalsize}
\subsubsectionfont{\normalsize}
\paragraphfont{\normalsize}
\newtheorem{thm}{Theorem}
\newtheorem*{thm*}{Theorem}
\newtheorem{lem}[thm]{Lemma}

\newtheorem{cor}[thm]{Corollary}
\newtheorem{clm}[thm]{Claim}
\newcommand{\llll}[1] {\left #1}
\newcommand{\rrrr}[1] {\right #1}

\newcommand{\dddd}[2]{\dfrac{#1}{#2}}

\newcommand{\aaaa}{\alpha}

\newcommand{\llllll}{\lambda}
\newcommand{\bbbb}{\beta}
\newcommand{\GGGG}{\Gamma}
\newcommand{\gggg}{\gamma}
\newcommand{\oooo}{\omega}
\newcommand{\zzzz}{\zeta}

\begin{document}
\nocite{*}
\title{{\bf \normalsize HIGHER-ORDER NUMERICAL SOLUTIONS OF THE  FRACTIONAL RELAXATION-OSCILLATION EQUATION USING FRACTIONAL INTEGRATION}}
\author{Yuri Dimitrov\\
Department of Applied Mathematics and Statistics \\
University of Rousse, Rousse  7017, Bulgaria\\
\texttt{ymdimitrov@uni-ruse.bg}}
\maketitle
\begin{abstract} 
In the present paper we derive the asymptotic expansion formula for  the trapezoidal approximation of the fractional integral. We use the expansion formula to obtain approximations for the fractional integral of order $\aaaa,1+\aaaa,2+\aaaa,3+\aaaa$ and $4+\aaaa$. The approximations are applied for computing the numerical solutions of the  fractional relaxation-oscillation equation.

\noindent
{\bf 2010 Math Subject Classification:} 26A33, 45J05, 65D30, 65R10.
\end{abstract}
\section{Introduction}
Fractional differential equations is a growing field of mathematics with applications in economics, physics, etc. The main approach to computing  numerical solutions of fractional differential equations uses an approximation for the fractional derivative or the fractional integral. In the present paper we derive approximations for the fractional integral and we apply the approximations for computing the numerical solution of the fractional relaxation-oscillation equation.

The fractional integral of order $\aaaa>0$ on the interval $[0,x]$ is defined as 
$$I^\aaaa y(x)=\dddd{1}{\GGGG(\aaaa)}\int_0^x (x-t)^{\aaaa-1} y(t)d t.$$
The Caputo  and Riemann-Liouville  derivatives  of order $n+\aaaa$, where  $n$ is a positive integer and $0<\aaaa<1$ are defined as
$$y^{(n+\aaaa)}(x)=D^{n+\aaaa} y(x)=\dddd{1}{\Gamma (1-\aaaa)}\int_0^x \dfrac{y^{(n+1)}(t)}{(x-t)^{\aaaa}}dt,$$
$$D_{*}^{n+\aaaa} y(x)=\dddd{1}{\Gamma (1-\aaaa)}\dddd{d^{n+1}}{dx^{n+1}}\int_0^x \dfrac{y(t)}{(x-t)^{\aaaa}}dt.$$
The Caputo and Riemann-Liouville derivatives of order $n+\aaaa$ are expressed as a composition of the fractional integral of order $1-\aaaa$ and the derivative of the function of order $n+1$
$$D^{n+\aaaa} y(x)=I^{1-\aaaa}D^{n+1} y(x),\quad D_*^{n+\aaaa} y(x)=D^{n+1}I^{1-\aaaa} y(x).$$
When $y(x)$ is a sufficiently differentiable function on the interval $[0,x]$, the Riemann-Liouville and Caputo derivatives are related as 
\begin{equation*}
D_{*}^{n+\aaaa} y(x)=D^{n+\aaaa} y(x)+
\sum_{k=0}^n \dfrac{y^{(k)}(0)}{\Gamma (k-\aaaa-n+1)}x^{k-\aaaa-n}.
\end{equation*}
The Riemann-Liouville and Caputo derivatives are equal, when the function $y$ satisfies the condition $y(0)=y'(0)=\cdots=y^{(n)}(0)=0$. The composition of fractional integrals and fractional derivatives has the following properties
\begin{equation}\label{2_1}
\begin{array}{ll
}
I^\aaaa I^\bbbb y(x)= I^\bbbb y(x) I^\aaaa y(x)=I^{\aaaa+\bbbb}y(x), &(\aaaa>0,\bbbb>0),\\
D^\aaaa I^\aaaa y(x)=y(x),  I^\aaaa D^\aaaa y(x)=y(x)-y(0), &(0<\aaaa<1),\\
D^\aaaa I^\aaaa y(x)=y(x),  I^\aaaa D^\aaaa y(x)=y(x)-y(0)-y'(0)x, &(1<\aaaa<2).
\end{array}
\end{equation}
When the $(n+1)$-st derivative of the function $y$  is bounded,  the Caputo derivative of order $n+\aaaa$ at the  point $x=0$ is equal to zero $y^{(n+\aaaa)}(0)=0$.
The fractional integral and the Riemann-Liouville derivative of the power function  satisfy
$$I^\aaaa x^p=\dddd{\GGGG(p+1)}{\GGGG(p+\aaaa+1)}x^{p+\aaaa},\quad D_*^\aaaa x^p=\dddd{\GGGG(p+1)}{\GGGG(p-\aaaa+1)}x^{p-\aaaa}\quad (p>-1).$$
The  Caputo derivative and the  fractional integral of the exponential function satisfy
$$D^\aaaa e^{\llllll x}=\llllll x^{1-\aaaa}E_{1,2-\aaaa}(\llllll x),\quad I^\aaaa e^{\llllll x}=x^{\aaaa}E_{1,1+\aaaa}(\llllll x),$$
where the two-parameter {\it Mittag-Leffler function} is defined  as
$$E_{\alpha,\beta} (x)=\sum_{n=0}^\infty \dfrac{x^n}{\Gamma(\alpha n+\beta)}.$$
The fractional integral and the fractional derivative are defined as integrals which have a singularity at the endpoint. The analytical and the numerical solutions of fractional differential equations involve special functions.

\noindent
The {\it Gamma function} is defined for $x>0$ as
$$\GGGG(x)=\int_0^\infty e^{-t} t^{x-1} d t.$$

\noindent
The logarithmic derivative of the gamma function is called {\it digamma function}.
$$\Psi(x)=\dddd{d}{dx}\ln \GGGG(x)=\dddd{\GGGG'(x)}{\GGGG(x)}.$$
The gamma and digamma functions satisfy the functional equations
$$\GGGG(1+x)=x\GGGG(x),\quad \Psi(x+1)=\Psi(x)+\dddd{1}{x}.$$

\noindent
The {\it Riemann zeta function} is defined for $x>1$ as
$$\zzzz(x)=1+\dddd{1}{2^x}+\dddd{1}{3^x}+\cdots+\dddd{1}{n^x}+\cdots=\sum_{n=1}^\infty {n^{-x}}.$$
The  Riemann zeta function  has representation 
$$\zzzz(x)=\dddd{1}{1-2^{1-\aaaa}}\sum_{n=1}^\infty \dddd{(-1)^{n-1}}{n^x},$$
when $x>0$,  and satisfies the functional equation
$$\zzzz(x)=2^x\pi^{x-1}\sin \llll(\dddd{\pi x}{2} \rrrr)\GGGG (1-x)\zzzz (1-x).$$
 The digamma function has a series representation
$$\Psi(x+1)=-\gggg-\sum_{n=1}^\infty \zzzz(k+1)(-x)^k,\qquad (|x|<1),$$
where $\gggg\approx 0.577$ is the Euler-Mascheroni constant  
$$\gggg=\lim_{n\rightarrow \infty} \llll( \sum_{k=1}^{n}\dddd{1}{k}-\ln n \rrrr).$$
The {\it polylogarithm function} is defined as 
$$Li_\aaaa(x)=\sum_{n=1}^\infty \dddd{x^n}{n^\aaaa},\quad (|x|<1),$$
and satisfies
$$Li_1(x)=-\ln(1-x),\quad Li_\aaaa(1)=\zzzz(\aaaa).$$
The function $Li_\aaaa\llll(e^{z}\rrrr)$ has  expansion 
\begin{equation}\label{4_1}
Li_\aaaa\llll(e^{z}\rrrr)=\GGGG(1-\aaaa) (-z)^{\aaaa-1}+\sum_{k=0}^\infty \dddd{\zzzz(\aaaa-k)}{k!} z^k,\qquad (|z|<2\pi).
\end{equation}
Let $x=n h$, where $n$ is a positive integer. Denote by $T_n[y]$ the trapezoidal sum of the function $y(t)$ on the interval $[0,x]$.
$$T_n[y]=\dddd{h}{2}\llll(\dddd{y(0)}{2}+\sum_{k=1}^{n-1}y(k h)+\dddd{y(x)}{2} \rrrr).$$
From the Euler-MacLaurin formula \cite{Kouba2013,Lampert2001}
\begin{equation}\label{4_2}
T_n[y]=\int_0^x y(t)dt+\sum_{k=1}^m \dddd{B_{2k}}{(2k)!}\llll(f^{(2k-1)}(x)- f^{(2k-1)}(0)\rrrr)h^{2k}+R_{2m}(h)h^{2m},
\end{equation}
where $B_k$ are the Bernoulli numbers $1,\pm1/2,1/6,0,-1/30,0,\cdots$. The error term of the Euler-MacLaurin formula has coefficient
$$R_{2m}(h)=\dddd{1}{(2m)!}\int_0^x\tilde{B}_{2m}(-   t/h)f^{(2m)}(t)dt,$$
where  $\tilde{B}_{n}(x)=B_n(\{x\})$ is the Bernoulli 1-periodic function
$$\tilde{B}_{n}(x)=B_n(\{x\}),\qquad (\{x\}=x-\lfloor  x \rfloor),$$
and $B_n(x)$ is the Bernoulli polynomial
$$B_n(x)=\sum_{k=0}^n \binom{n}{k}B_{n-k}x^k=B_0x^n+n B_1x^{n-1}+\cdots+B_n.$$
 The coefficient $R_{2m}(h)$ satisfies $\lim_{h\rightarrow 0} R_{2m}(h)=0$.
The Bernoulli numbers $B_n$ with an odd index are equal to zero. The Bernoulli numbers with an even index are expressed with the values of the Riemann zeta function 
$$\dddd{B_{2n}}{(2n)!}=(-1)^{n+1}\dddd{2 \zzzz(2n)}{(2\pi)^{2n}}.$$
The Bernoulli numbers are called first Bernoulli numbers when $B_1 =-1/2$ and  Bernoulli numbers of the second kind if $B_1 = 1/2$. The generating function of the  first Bernoulli numbers satisfies
\begin{equation}\label{5_0}
\dddd{z e^{n z}}{e^z-1}=\sum_{m=0}^\infty B_m(n)\dddd{z^m}{m!}.
\end{equation}
Denote by $J^\aaaa y(x)$ and $K^\aaaa y(x)$ the fractional integrals
$$J^\aaaa y(x)=\int_0^x (x-t)^\aaaa y(t)dt,\quad K^\aaaa y(x)=\int_0^x \dddd{ y(t)}{(x-t)^{1-\aaaa}}dt.$$
The  fractional integrals $J^\aaaa y(x)$ and $K^\aaaa y(x)$ are related to $I^\aaaa y(x)$ as
$$ J^{\aaaa} y(x)=\GGGG(1+\aaaa) I^{1+\aaaa} y(x),\quad  K^{\aaaa} y(x)=\GGGG(\aaaa) I^{\aaaa} y(x).$$
In previous work \cite{Dimitrov2015_2} we derived the fourth-order expansion of the trapezoidal approximation for the fractional integral $J^\aaaa y(x)$
\begin{align}\label{5_1}
h^{1+\aaaa}\sum_{k=1}^{n-1} k^{\aaaa} y(x-k h)&+\dddd{y(0)x^\aaaa}{2}h=J^{\aaaa}y(x)+\dddd{\aaaa x^{\aaaa-1}y(0)-x^\aaaa y'(0)}{12}h^2\nonumber\\
&+\zzzz(-\aaaa) y(x)h^{1+\aaaa}-y'(x)\zzzz(-1-\aaaa)h^{2+\aaaa}\\
&+\dddd{y''(x)}{2}\zzzz(-2-\aaaa)h^{3+\aaaa}+O\llll(h^4 \rrrr).\nonumber
\end{align}
In section 2 we use \eqref{5_1} to obtain the fourth-order expansion \eqref{7_1} of the trapezoidal approximation for the fractional integral $K^\aaaa y(x)$. In section 3 we use the Fourier transform method to derive the asymptotic expansion formula for the left Riemann sums of the fractional integral $K^\aaaa y(x)$
\begin{align*}  h^{\aaaa} \sum_{k=1}^{n-1} &\dddd{y(x-k h)}{k^{1-\aaaa}}=\int_0^x \dddd{y(t)}{(x-t)^{1-\aaaa}}dt +\sum_{k=0}^\infty (-1)^k\dddd{\zzzz(1-\aaaa-k)}{k!}y^{(k)}(x) h^{k+\aaaa}\\
&-\GGGG(\aaaa)\sum_{k=0}^\infty \dddd{B_{k+1}}{(k+1)!}\llll( \sum_{m=0}^k (-1)^m\binom{k}{m}\dddd{x^{\aaaa-m-1}}{\GGGG(\aaaa-m)}y^{(m-k)}(0) \rrrr)h^{k+1}.
\end{align*}
The left Riemann sums approximate the fractional integral with accuracy $O\llll(h^\aaaa\rrrr)$. When the solution of the fractional relaxation-oscillation equation \eqref{16_2} satisfies $y(0)=y'(0)=0$ it is equivalent to the fractional  integral equation \eqref{27_0}. In section 4 we discuss the convergence of the numerical solution  of equation  \eqref{27_0}. In sections 5 we derive approximations \eqref{27_2},\eqref{29_1},\eqref{30_1} and \eqref{32_1} for the fractional integral $I^\aaaa y(x)$ of orders $1+\aaaa,2+\aaaa,3+\aaaa,4+\aaaa$ and we apply the approximations for computing the numerical solution of the fractional integral equation \eqref{27_0}.
\section{Fourth Order Approximation of the Fractional Integral}
In this section we use  approximation \eqref{5_1} for the fractional integral $J^\aaaa y(x)$ and the formula for the asymptotic expansion of the sum of powers \eqref{6_1}  to derive the formula for the fourth order expansion of the trapezoidal approximation for the fractional integral $K^\aaaa y(x)$.
\begin{equation}\label{6_1}
\sum_{k=1}^{n-1}k^\aaaa=\zzzz(-\aaaa)+\dddd{n^{1+\aaaa}}{1+\aaaa}
\sum_{m=0}^{\infty}\binom{\aaaa+1}{m}\dddd{B_m}{n^m}.
\end{equation}
 Denote
$$T_n[y]=h(y(0)/2+y(h)+y(2h)+\dots+y((n-1)h)),$$
$$T_n[y]=h^\aaaa\llll(1+2^\aaaa+\cdots+(n-1)^{\aaaa-1}\rrrr)+\dddd{h^\aaaa n^{\aaaa-1}}{2}.$$
The definite integral of the function $y(t)=(x-t)^{\aaaa-1}$ is equal to the fractional integral $K^\aaaa$ of the constant function $1$.
$$Iy(t)=K^\aaaa 1=\int_0^x(x-t)^{\aaaa-1}dt=\dddd{x^\aaaa}{\aaaa}.$$
We consider $T_n[(x-t)^{\aaaa-1}]$ as the trapezoidal approximation for the function $y(t)=(x-t)^{\aaaa-1}$ because the function $y(x)=0$ when $\aaaa>1$ and $y(x)$ is undefined when $0<\aaaa<1$.
\begin{lem} Let $y(t)=(x-t)^{\aaaa-1}$. Then
$$T_n[y]=\int_0^x(x-t)^{\aaaa-1}dt+\zzzz(1-\aaaa)h^\aaaa+\dddd{\aaaa-1}{12x^{2-\aaaa}}h^2+O\llll(h^4\rrrr).$$
\end{lem}
\begin{proof} From the formula for the sum of powers \eqref{6_1}
$$\sum_{k=1}^{n-1} k^{\aaaa-1}=\zzzz(1-\aaaa)+\dddd{n^\aaaa}{\aaaa}-\dddd{n^{\aaaa-1}}{2}+\dddd{\aaaa-1}{12}n^{\aaaa-2}+O\llll(\dddd{1}{N^{4-\aaaa}}\rrrr).
$$
Then
$$T_n[y]=h^\aaaa\llll(\zzzz(1-\aaaa)+\dddd{n^\aaaa}{\aaaa}-\dddd{n^{\aaaa-1}}{2}+\dddd{\aaaa-1}{12}n^{\aaaa-2}+O\llll(\dddd{1}{n^{4-\aaaa}}\rrrr)\rrrr)+\dddd{h^\aaaa n^{\aaaa-1}}{2},$$
$$T_n[y]=\dddd{(n h)^\aaaa}{\aaaa}+\zzzz(1-\aaaa)h^\aaaa+\dddd{\aaaa-1}{12}n^{\aaaa-2}h^\aaaa+O\llll(\dddd{h^\aaaa}{n^{4-\aaaa}}\rrrr),$$
$$T_n[y]=\int_0^x(x-t)^{\aaaa-1}dt+\zzzz(1-\aaaa)h^\aaaa+\dddd{\aaaa-1}{12 x^{2-\aaaa}}h^2+O\llll(\dddd{h^4}{x^{4-\aaaa}}\rrrr).$$
\end{proof}
In Theorem 2 and Corollary 3  we derive the fourth-order expansion of the trapezoidal approximation for the fractional integral. 
\begin{thm} Let $y(t)$ be a polynomial. Then
\begin{align}\label{7_1}
h^{\aaaa}\sum_{k=1}^{n-1} \dddd{ y(x-kh)}{k^{1-\aaaa}}&+\dddd{y(0)}{2x^{1-\aaaa}}h=K^{\aaaa}y(x)+\dddd{(\aaaa-1) x^{\aaaa-2}y(0)-x^{\aaaa-1} y'(0)}{12}h^2\nonumber\\
+\zzzz&(1-\aaaa) y(x)h^{\aaaa}-\zzzz(-\aaaa)y'(x)h^{1+\aaaa}\nonumber
+\zzzz(-1-\aaaa)\dddd{y''(x)}{2}h^{2+\aaaa}\nonumber\\
-\zzzz&(-2-\aaaa)\dddd{y'''(x)}{6}h^{3+\aaaa}+O\llll(h^4 \rrrr).
\end{align}
\end{thm}
\begin{proof} Let $y(t)$ be a polynomial of degree $m$.
$$y(t)=p_0+p_1 (x-t)+p_2 (x-t)^2+\cdots+p_m (x-t)^m.$$
Denote by $q(t)$ the polynomial 
$$q(t)=p_1+p_2 (x-t)+p_3 (x-t)^2+\cdots+p_m (x-t)^{m-1}.$$
The coefficients $p_0,p_1,p_2$ and $p_3$ are expressed with the values of the functions $y(x)$ and $q(x)$ and their derivatives as
\begin{align}\label{8_0}
&p_0=y(x),\; p_1=-y'(x)=q(x),\; p_2=\dddd{y''(x)}{2}=-q'(x),\nonumber\\
& p_3=-\dddd{y'''(x)}{6}=\dddd{q''(x)}{2}.
\end{align}
The fractional integrals $K^\aaaa y(x)$ and $J^\aaaa q(x)$ satisfy
$$K^\aaaa y(x)=p_0\int_0^x (x-t)^{\aaaa-1}dt+J^\aaaa q(x).$$
From the formula for the fourth-order approximation \eqref{5_1} of the fractional integral $J^\aaaa q(x)$.
\begin{align}\label{8_1}
h^{1+\aaaa}\sum_{k=1}^{n-1} &k^\aaaa q(x-k h)+\dddd{q(0)x^\aaaa}{2}h=J^\aaaa q(x)+
\llll(\aaaa x^{\aaaa-1}q(0)-x^\aaaa q'(0)\rrrr)\dddd{h^2}{12}\\
&+p_1\zzzz(-\aaaa)h^{1+\aaaa}+p_2\zzzz(-1-\aaaa)h^{2+\aaaa}+p_3\zzzz(-2-\aaaa)h^{3+\aaaa}+O\llll(h^4\rrrr).\nonumber
\end{align}
The functions $y(t)$ and $q(t)$ satisfy
\begin{equation}\label{8_2}
y(t)-y(x)=q(t)(x-t).
\end{equation}
Then
$$q(0)=\dddd{y(0)-y(x)}{x}.$$
By differentiating \eqref{8_2} 
$$ y'(t)=q'(t)(x-t)-q(t),$$
$$y'(0)=q'(0)x-q(0),$$
$$q'(0)=\dddd{x y'(0)+y(0)-y(x)}{x^2}.$$
Hence
\begin{align}\label{9_0}
\aaaa x^{\aaaa-1}q(0)-x^\aaaa q'(0)=x^{\aaaa-2}\llll( (1-\aaaa)y(x)+(\aaaa-1)y(0)-x y'(0)  \rrrr).
\end{align}
From \eqref{8_2} : $y(x-k h)=p_0+k h q(x- k h)$. Then
\begin{align}\label{9_1}
h^{\aaaa}\sum_{k=1}^{n-1}& \dddd{y(x-k h)}{k^{1-\aaaa}}+\dddd{y(0)}{2x^{1-\aaaa}}h=h^{\aaaa}\sum_{k=1}^{n-1} \dddd{p_0+k h q(x- k h)}{k^{1-\aaaa}}+\dddd{p_0+x q(0)}{2x^{1-\aaaa}}\\
&=\llll(h^{1+\aaaa}\sum_{k=1}^{n-1} k^\aaaa q(x-k h)+\dddd{q(0)x^\aaaa}{2}\rrrr)+p_0 h^\aaaa \llll(\sum_{k=1}^{n-1} k^{\aaaa-1} +\dddd{n^{\aaaa-1} }{2}  \rrrr).\nonumber
\end{align}
From  Lemma 1
\begin{equation}\label{9_2}
h^\aaaa \llll(\sum_{k=1}^{n-1} k^{\aaaa-1} +\dddd{n^{\aaaa-1} }{2}  \rrrr)=K^\aaaa 1+\zzzz(1-\aaaa)h^\aaaa+\dddd{\aaaa-1}{12x^{2-\aaaa}}h^2+O\llll(h^4\rrrr).
\end{equation}
From \eqref{9_1} and \eqref{9_2}
\begin{align}\label{9_3}
h^{\aaaa}\sum_{k=1}^{n-1} \dddd{y(x-k h)}{k^{1-\aaaa}}&+\dddd{y(0)}{2x^{1-\aaaa}}h=K^\aaaa y(x)+
\llll(\aaaa x^{\aaaa-1}q(0)-x^\aaaa q'(0)\rrrr)\dddd{h^2}{12}\nonumber\\
&+p_0\zzzz(1-\aaaa)h^{\aaaa}+p_1\zzzz(-\aaaa)h^{1+\aaaa}+p_2\zzzz(-1-\aaaa)h^{2+\aaaa}\nonumber\\
&+p_3\zzzz(-2-\aaaa)h^{3+\aaaa}+O\llll(h^4\rrrr).
\end{align}
The formula for the fourth-order expansion \eqref{7_1} of the trapezoidal approximation for the fractional integral $K^\aaaa y(x)$ follows from \eqref{8_0},\eqref{9_0} and \eqref{9_3}.
\end{proof}
From the Stone-Weierstrass Theorem, every continuous and differentiable function and its derivatives are uniform limits of polynomials.
\begin{cor} The statement of Theorem 2 holds for all functions $y\in C^4[0,x]$.
\end{cor}
The proof of Corollary 3 is similar to the proof of Theorem 8 in \cite{Dimitrov2015_1}.
\begin{table}[ht]
    \caption{Error and order of approximation \eqref{7_1} for the fractional integral $K^\aaaa y(x)$ for  $y(t)=e^t,\aaaa=0.5,x=2$ (left) and $y(t)=\ln (t+3),\aaaa=0.25,$ $x=1$ (right).}
    \begin{subtable}{0.5\linewidth}
      \centering
  \begin{tabular}{l c c }
  \hline \hline
    $\boldsymbol{h}$ & $\mathbf{Error}$ & $\mathbf{Order}$  \\ 
		\hline \hline
$0.025$        &$1.06\times 10^{-9}$      &$4.04725$\\
$0.0125$       &$6.48\times 10^{-11}$     &$4.03414$\\
$0.00625$      &$3.98\times 10^{-12}$     &$4.02694$\\
$0.003125$     &$2.34\times 10^{-13}$     &$4.08360$\\
		\hline
  \end{tabular}
    \end{subtable}%
    \begin{subtable}{.5\linewidth}
      \centering
				\;
  \begin{tabular}{ l  c  c }
    \hline \hline
    $\boldsymbol{h}$ & $\mathbf{Error}$ &$\mathbf{Order}$  \\ \hline \hline
$0.025$        &$2.77\times 10^{-9}$      &$3.99877$\\
$0.0125$       &$1.73\times 10^{-10}$     &$3.99963$\\
$0.00625$      &$1.08\times 10^{-11}$     &$3.99977$\\
$0.003125$     &$6.78\times 10^{-13}$     &$3.99562$\\
\hline
  \end{tabular}
    \end{subtable} 
\end{table}
\begin{cor}Let $y(0)=y'(0)=y''(0)=y'''(0)=0$. Then
\begin{align}\label{9_4}
h^{\aaaa}\sum_{k=1}^{n-1} \dddd{ y(x-kh)}{k^{1-\aaaa}}&=K^{\aaaa}y(x)+\zzzz(1-\aaaa) y(x)h^{\aaaa}-\zzzz(-\aaaa)y'(x)h^{1+\aaaa}\\
+&\zzzz(-1-\aaaa)\dddd{y''(x)}{2}h^{2+\aaaa}\nonumber
-\zzzz(-2-\aaaa)\dddd{y'''(x)}{6}h^{3+\aaaa}+O\llll(h^{4+\aaaa} \rrrr).
\end{align}
\end{cor}
\section{Expansion Formula for the Trapezoidal Approximation of the Fractional Integral}
In this section we use the Fourier transform method to derive the asymptotic Euler-MacLaurin formula \eqref{11_1} and the expansion formula for the trapezoidal approximation of the fractional integral \eqref{15_1}. The exponential Fourier transform is defined as
$$\mathcal{F}[y(x)](\oooo)=\hat{y}(\oooo)=\int_{-\infty}^\infty e^{i \oooo t}y(t) d t.$$
The function $y(x)=x^{\aaaa-1}\chi_{[0,\infty)}$ has Fourier transform $\hat{y}(w)=\GGGG(\aaaa)(-i\oooo)^\aaaa$, where $\chi_{[0,\infty)}$ is the characteristic function of the interval $[0,\infty)$. The $n$-th derivative of the function $y$ has Fourier transform $\mathcal{F}[y^{(n)}(x)](\oooo)=(-i \oooo)^n \hat{y}$. The Fourier transform has the translation and the convolution properties
$$\mathcal{F}[y(x-a)](w)=e^{i \oooo a} \hat{y}(w),\quad \mathcal{F}[y(x)*z(x)](\oooo)= \hat{y}(\oooo) \hat{z}(\oooo).$$
The Caputo derivatives and the fractional integral are defined as the convolution of the power function and the derivatives of the function $y$. 
$$\mathcal{F}[D^\aaaa y(x)](\oooo)=(-i \oooo)^\aaaa \hat{y}(\oooo),\quad \mathcal{F}[I^\aaaa y(x)](w)=(-i \oooo)^{-\aaaa} \hat{y}(\oooo).$$
The Fourier transform (generating function) method is used by Tadjeran et. al. \cite{TadjeranMeerschaertScheffer2006}, Ding and Li \cite{DingLi2016}, Tian et. al. \cite{TianZhouDeng2012} for constructing approximations for the fractional derivative related to the Gr\"unwald formula approximation.
\subsection{Euler-MacLaurin Formula}
By letting $m\rightarrow \infty$ in the Euler-MacLaurin formula \eqref{4_2} we obtain the series expansion of the trapezoidal approximation for the definite integral.
\begin{equation}\label{11_1}
T_n[y]=\int_0^x y(\xi)d\xi+\sum_{n=1}^\infty \dddd{B_{2n}}{(2n)!}\llll(y^{(2n-1)}(x) -y^{(2n-1)}(0)\rrrr)h^{2n}.
\end{equation}
Denote 
$$S_n[y]=h \sum_{k=0}^n y(x-k h).$$
The approximation $S_n[y]$ for the definite integral has generating function 
$$G(z)=\sum_{k=0}^\infty z^k=\dddd{1}{1-z}.$$
Set $n=0$ in \eqref{5_0}
$$\dddd{z }{e^z-1}=\sum_{m=0}^\infty \dddd{B_m}{m!}z^m,  $$
\begin{equation}\label{11_2}
\dddd{e^z}{e^z-1}=1+\dddd{1}{e^z-1}=1+\sum_{m=0}^\infty \dddd{B_m}{m!}z^{m-1}=\sum_{m=0}^\infty \dddd{B_m}{m!}z^{m-1}.
\end{equation}
where $B_m$ are the second Bernoulli numbers ($B_1=1/2$).
By applying Fourier transform to $S_n[y]$ and letting $n\rightarrow \infty$ we obtain
\begin{align*}
\mathcal{F}[S_\infty [y]](\oooo)&=h \sum_{k=0}^\infty \mathcal{F}[y(x-k h)](\oooo)=h \sum_{k=0}^\infty e^{i k h \oooo}\hat{y}(\oooo)\\
&=h \hat{y}(\oooo)\sum_{k=0}^\infty \llll( e^{i  h \oooo}\rrrr)^k=\dddd{h \hat{y}(\oooo)}{1-e^{i  h \oooo}}=\dddd{e^{-i  h \oooo}}{e^{-i  h \oooo}-1}h \hat{y}(\oooo).
\end{align*} 
From \eqref{11_2}
\begin{align}\label{11_3}
\mathcal{F}[S_\infty [y]](\oooo) &= h \hat{y}(\oooo)\sum_{m=0}^\infty \dddd{B_m}{m!}(-i h \oooo)^{m-1}\\
&=h  \llll(\hat{y}(\oooo)(-i h \oooo)^{-1}+\dddd{1}{2}\hat{y}(\oooo)+\sum_{m=1}^\infty \dddd{B_{2m}}{(2m)!}(-i h \oooo)^{2 m-1}\hat{y}(\oooo)\rrrr)\nonumber\\
&=\hat{y}(\oooo)(-i  \oooo)^{-1}+\dddd{h}{2}\hat{y}(\oooo)+\sum_{m=1}^\infty \dddd{B_{2m}}{(2m)!}(-i  \oooo)^{2 m-1}\hat{y}(\oooo)h^{2m}.\nonumber
\end{align}
Denote by
$$\mathcal{R}\llll( S_n [y]-\int_0^x y(t)dt \rrrr),\quad \mathcal{L}\llll( S_n [y]-\int_0^x y(t)dt \rrrr),$$
the left and right endpoint sums of the expansion formula of the approximation $S_n[y]$ for the definite integral. The right endpoint sum is the expansion formula for the functions whose derivatives are equal to zero at the left limit  $y^{(k)}(0)=0,\;k=0,1,\cdots$. The left endpoint sum corresponds to the functions whose derivatives vanish at the right limit of the definite integral $y^{(k)}(x)=0,\;k=0,1,\cdots$.
The Fourier transform method yields the right endpoint asymptotic sum. By applying inverse Fourier transform to \eqref{11_3} we obtain
\begin{equation}\label{12_1}
\mathcal{R}\llll( S_n [y]-\int_0^x y(t)dt \rrrr)=\dddd{h}{2}y(x)+\sum_{m=1}^\infty \dddd{B_{2m}}{(2m)!}y^{(2m-1)}(x)h^{2m}.
\end{equation}
The asymptotic formula for the left endpoint is determined from the substitution $z(t)=y(x-t)$. The function $z(t)$ satisfies $S_n[z]=S_n[y]$ and $z(x)=y(0),\;z^{(2m-1)}(x)=-y^{(2m-1)}(0)$. The left endpoint asymptotic sum of the function $y(t)$ is equal to the right endpoint sum of the function $z(t)$. 
\begin{equation}\label{12_2}
\mathcal{L}\llll( S_n [y]-\int_0^x y(t)dt \rrrr)=\dddd{h}{2}y(0)-\sum_{m=1}^\infty \dddd{B_{2m}}{(2m)!}y^{(2m-1)}(0)h^{2m}.
\end{equation}
By combining \eqref{12_1} and \eqref{12_2}
$$ S_n [y]=\int_0^x y(t)dt +\dddd{h}{2}(y(x)+y(0))+\sum_{m=1}^\infty \dddd{B_{2m}}{(2m)!}\llll( y^{(2m-1)}(x)-y^{(2m-1)}(0)\rrrr)h^{2m}.$$
The Euler-MacLaurin formula \eqref{11_1} is obtained from\eqref{12_1}, \eqref{12_2} and the relation between the approximations $S_n[y]$ and $T_n[y]$ for the definite integral.
$$T_n [y]=S_n [y]-\dddd{h}{2}(y(x)+y(0)).$$
The partial sums of the Euler-MacLaurin formula are used to derive high-order approximations for the definite integral.
\subsection{Trapezoidal Approximation of the Fractional Integral} 
We apply the Fourier transform method to derive the asymptotic expansion formula \eqref{15_1} for the trapezoidal approximation of the fractional integral. Let  $z(t)=y(t)/(x-t)^{1-a}$. Denote by $S_n[y]$ the left Riemann sum of the fractional integral $K^\aaaa y(x)$
$$S_n[y]=h\sum_{k=1}^n z(x-k h)=h^{\aaaa} \sum_{k=1}^{n} \dddd{y(x-k h)}{k^{1-\aaaa}}.$$
The generating function of $S_n[y]$ is
$$G(z)=\sum_{k=1}^\infty \dddd{z^k}{k^{1-\aaaa}}=Li_{1-\aaaa}(z).$$
From \eqref{4_1}
$$Li_{1-\aaaa}\llll(e^{z}\rrrr)=\GGGG(\aaaa) (-z)^{-\aaaa}+\sum_{k=0}^\infty \dddd{\zzzz(1-\aaaa-k)}{k!}z^k.$$
By applying Fourier transform to $S_n[y]$ and letting $n\rightarrow \infty$ we obtain
\begin{align*}
\mathcal{F}[S_\infty [y]](\oooo)&=h^\aaaa \sum_{k=1}^\infty \dddd{\mathcal{F}[y(x-k h)](\oooo)}{k^{1-\aaaa}}=h^\aaaa \sum_{k=1}^\infty \dddd{e^{i k h \oooo}}{k^{1-\aaaa}}\hat{y}(\oooo)\\
&=h^\aaaa\hat{y}(\oooo)\sum_{k=1}^\infty \dddd{ \llll( e^{i  h \oooo}\rrrr)^k}{k^{1-\aaaa}}=h^\aaaa Li_{1-\aaaa}\llll(e^{i  h \oooo}\rrrr)\hat{y}(\oooo),
\end{align*}
\begin{align*}
\mathcal{F}[S_\infty [y]](\oooo) &= h^\aaaa \hat{y}(\oooo)\llll(\GGGG(\aaaa) (-i h\oooo)^{-\aaaa}+\sum_{k=0}^\infty \dddd{\zzzz(1-\aaaa-k)}{k!}(i h \oooo)^k\rrrr)\\
&=\GGGG(\aaaa) (-i \oooo)^{-\aaaa}+\sum_{k=0}^\infty (-1)^k\dddd{\zzzz(1-\aaaa-k)}{k!}(-i \oooo)^k h^{k+\aaaa}.
\end{align*}
By applying inverse Fourier transform we obtain the asymptotic sum for the right endpoint of the approximation $S_n [y]$ for the fractional integral $K^\aaaa y(x)$. 
\begin{equation}\label{13_1}
\mathcal{R}\llll( S_n [y]-\int_0^x \dddd{y(t)}{(x-t)^{1-\aaaa}}dt \rrrr)=\sum_{k=0}^\infty (-1)^k\dddd{\zzzz(1-\aaaa-k)}{k!}y^{(k)}(x) h^{k+\aaaa}.
\end{equation}
The asymptotic formula for the left endpoint $t=0$ is obtained from the Euler-MacLaurin formula for the function  $z(t)=(x-t)^{\aaaa-1}y(t)$. 
\begin{equation}\label{14_1}
\mathcal{L}\llll( S_n [y]-\int_0^x \dddd{y(t)}{(x-t)^{1-\aaaa}}dt \rrrr)=-\sum_{k=0}^\infty \dddd{B_{k+1}}{(k+1)!}z^{(k)}(0)h^{k+1}.
\end{equation}
The function $(x-t)^{\aaaa-1}$ has  derivative of order $m$
$$\dddd{d^m}{d t^m}\llll((x-t)^{\aaaa-1} \rrrr)=(-1)^m(\aaaa-1)\cdots (\aaaa-m)(x-t)^{\aaaa-m-1},$$
$$\dddd{d^m}{d t^m}\llll((x-t)^{\aaaa-1} \rrrr)=(-1)^m\dddd{\GGGG(\aaaa)}{\GGGG(\aaaa-m)}(x-t)^{\aaaa-m-1}.$$
From the Leibnitz differentiation formula
$$z^{(k)}(t)=\sum_{m=0}^k \binom{k}{m}y^{(m-k)}(t)\dddd{d^m}{d t^m}\llll((x-t)^{\aaaa-1} \rrrr),$$
$$z^{(k)}(t)=\GGGG(\aaaa)\sum_{m=0}^k (-1)^m\binom{k}{m}\dddd{(x-t)^{\aaaa-m-1}}{\GGGG(\aaaa-m)}y^{(m-k)}(t).$$
Then
$$z^{(k)}(0)=\GGGG(\aaaa)\sum_{m=0}^k (-1)^m\binom{k}{m}\dddd{x^{\aaaa-m-1}}{\GGGG(\aaaa-m)}y^{(m-k)}(0).$$
From \eqref{14_1}
\begin{align}\label{14_15}
\mathcal{L}( S_n [y]-&K^\aaaa y(x) )\\
&=-\GGGG(\aaaa)\sum_{k=0}^\infty \dddd{B_{k+1}}{(k+1)!}\llll( \sum_{m=0}^k (-1)^m\binom{k}{m}\dddd{x^{\aaaa-m-1}}{\GGGG(\aaaa-m)}y^{(m-k)}(0) \rrrr)h^{k+1}.\nonumber
\end{align}
By combining \eqref{13_1} and \eqref{14_15}  we obtain 
\begin{align}\label{14_2}
  h^{\aaaa} &\sum_{k=1}^{n} \dddd{y(x-k h)}{k^{1-\aaaa}}=\int_0^x \dddd{y(t)}{(x-t)^{1-\aaaa}}dt +\sum_{k=0}^\infty (-1)^k\dddd{\zzzz(1-\aaaa-k)}{k!}y^{(k)}(x) h^{k+\aaaa}\nonumber\\
&-\GGGG(\aaaa)\sum_{k=0}^\infty \dddd{B_{k+1}}{(k+1)!}\llll( \sum_{m=0}^k (-1)^m\binom{k}{m}\dddd{x^{\aaaa-m-1}}{\GGGG(\aaaa-m)}y^{(m-k)}(0) \rrrr)h^{k+1}.
\end{align}
The term corresponding to $k=0$ in the second sum is $y(0)x^{\aaaa-1}h/2$. Let $T_n[y]$ be the trapezoidal approximation for the fractional integral $I^\aaaa y(x)$
$$T_n[y]=\dddd{1}{\GGGG(\aaaa)}\llll(h^{\aaaa}\sum_{k=1}^{n-1} \dddd{ y(x-kh)}{k^{1-\aaaa}}+\dddd{y(0)}{2x^{1-\aaaa}}h\rrrr).$$
The trapezoidal approximation $T_n[y]$ and the left Riemann sum $S_n[y]$ for the fractional integral $K^\aaaa y(x)$ satisfy
$$T_n[y]=\dddd{1}{\GGGG(\aaaa)}\llll(S_n[y]-\dddd{y(0)}{2x^{1-\aaaa}}h\rrrr).$$
From \eqref{14_2} we obtain the  expansion formula for the trapezoidal approximation for the fractional integral $I^\aaaa y(x)$
\begin{align} \label{15_1} T_n[y]=I^\aaaa &y(x) +\dddd{1}{\GGGG(\aaaa)}\sum_{k=0}^\infty (-1)^k\dddd{\zzzz(1-\aaaa-k)}{k!}y^{(k)}(x) h^{k+\aaaa}\\
&-\sum_{k=1}^\infty \dddd{B_{k+1}}{(k+1)!}\llll( \sum_{m=0}^k (-1)^m\binom{k}{m}\dddd{x^{\aaaa-m-1}}{\GGGG(\aaaa-m)}y^{(m-k)}(0) \rrrr)h^{k+1}.\nonumber
\end{align}
In  section 2 we derived the fourth-order approximation \eqref{7_1} for the fractional integral $K^\aaaa y(x)$. From \eqref{15_1} we obtain the sixth-order approximation
\begin{align*}
h^{\aaaa}\sum_{k=1}^{n-1} \dddd{ y(x-kh)}{k^{1-\aaaa}}+\dddd{y(0)}{2x^{1-\aaaa}}h&=K^{\aaaa}y(x)-\dddd{1}{12}\llll((1-\aaaa) x^{\aaaa-2}y(0)+x^{\aaaa-1} y'(0)\rrrr)h^2\nonumber\\
+\zzzz(1-\aaaa) y(x)h^{\aaaa}-&\zzzz(-\aaaa)y'(x)h^{1+\aaaa}\nonumber
+\zzzz(-1-\aaaa)\dddd{y''(x)}{2}h^{2+\aaaa}\nonumber\\
-\zzzz(-2-\aaaa)\dddd{y'''(x)}{6}h^{3+\aaaa}+&\zzzz(-3-\aaaa)\dddd{y^{(4)}(x)}{24}h^{4+\aaaa}-\zzzz(-4-\aaaa)\dddd{y^{(5)}(x)}{120}h^{5+\aaaa}\\
+\dddd{1}{720}\big((3-a)(2-a)&(1-\aaaa) x^{\aaaa-4}y(0)+3(2-a)(1-\aaaa) x^{\aaaa-3}y'(0)\\
&  +3(1-\aaaa) x^{\aaaa-2}y''(0)+x^{\aaaa-1} y'''(0)\big)h^4+O\llll(h^6\rrrr).
\end{align*}
\section{Numerical Solution of the  Relaxation-Oscillation Equation}
When $0<\aaaa<1$ the Caputo derivative is defined as the convolution of the first derivative and the power function $x^{-\aaaa}/\GGGG(1-\aaaa)$.
$$y^{(\aaaa)}(x)=\dddd{1}{\Gamma (1-\aaaa)}\int_0^x \dfrac{y'(\xi)}{(x-\xi)^{\aaaa}}d\xi.$$
When $1<\aaaa<2$ the Caputo derivative of the function $y(x)$ is defined as
$$y^{(\aaaa)}(x)=\dddd{1}{\Gamma (2-\aaaa)}\int_0^x \dfrac{y''(\xi)}{(x-\xi)^{\aaaa-1}}d\xi.$$
The Laplace transform of the derivatives of the Mittag-Leffler functions satisfies \cite{Podlubny1999}.  
$$\mathcal{L}\{ t^{\aaaa k+\bbbb-1} E^{(k)}_{\aaaa,\bbbb}(\pm a t^\aaaa) \}(s)=
\dddd{k! s^{\aaaa-\bbbb}}{\llll( s^\aaaa \mp a\rrrr)^{k+1}}.$$
The ordinary fractional differential equation
\begin{equation}\label{16_1}
y^{(\aaaa)}(x)+D y(x)=f(x)
\end{equation}
is called fractional relaxation equation when $0<\aaaa<1$ and fractional oscillation equation when  $1<\aaaa<2$. The analytical and numerical solutions  solutions of the fractional relaxation-oscillation equation are discussed in \cite{Diethelm2010, Dimitrov2014,Dimitrov2015_1,GulsuOzturkAnapali2013,Mainardi1996,Podlubny1999}. The exact solution of the fractional relaxation equation with initial condition $y(0)=y_0$ is determined with the Laplace transform method
$$y(x)= y_0 E_{\aaaa}(-D x^\aaaa)+\int_{0}^{x}\xi^{\aaaa-1} E_{\aaaa,\aaaa}\llll(-D \xi^\aaaa\rrrr)F(x-\xi)d\xi.$$
 The  fractional oscillation equation  has exact solution
$$y(x)= y_0 E_{\aaaa}(-D x^\aaaa)+y'_0 t E_{\aaaa,2}(-D x^\aaaa)+\int_{0}^{x}\xi^{\aaaa-1} E_{\aaaa,\aaaa}\llll(-D \xi^\aaaa\rrrr)F(x-\xi)d\xi,$$
where  $ y(0)=y_0,y'(0)=y'_0.$
\subsection{Numerical Solution of Order $\boldsymbol{\aaaa}$}
In section 2 we derived the fourth-order approximation \eqref{7_1} of the fractional integral. When $0<\aaaa<2$ and the function $y$ satisfies $y(0)=y'(0)=0$, the left Riemann sum is an approximation of the fractional integral of order $\aaaa$.
\begin{equation}\label{16_15}
\dddd{h^\aaaa}{\GGGG(\aaaa)} \sum_{k=1}^{n} \dddd{y_{n-k}}{k^{1-\aaaa}}=I^\aaaa y_n+O\llll( h^\aaaa \rrrr).
\end{equation}
Now we compute the numerical solution of the  relaxation-oscillation equation  
\begin{equation}\label{16_2}
y^{(\aaaa)}(x)+y(x)=f(x)
\end{equation}
 on the interval $[0,1]$, with the additional assumption that the solution satisfies $y(0)=y'(0)=0$. From  \eqref{2_1} the composition of the fractional integral and the Caputo derivative satisfies 
$$I^\aaaa D^\aaaa y(x)=D^\aaaa I^\aaaa y(x)=y(x).$$
By applying fractional integration of order $\aaaa$ to equation \eqref{16_2} we obtain
$$I^\aaaa y^{(\aaaa)}(x)+I^\aaaa y(x)=I^\aaaa f(x).$$
Denote $F(x)=I^\aaaa f(x)$.
\begin{equation}\label{17_1}
y(x)+I^\aaaa y(x)=F(x).
\end{equation}
The numerical solutions of fractional integro-differential equations are discussed by  Diethelm and Ford \cite{DiethelmFord2012}, Eslahchi et. al. \cite{EslahchiDehghanParvizi2014}, Mokhtary \cite{Mokhtary2016}.
 Let $h=1/n$, where $n$ is a positive integer. By approximating the fractional integral at the point $x=x_m=m h$ with \eqref{16_15} we obtain
\begin{equation}\label{17_15}
y_m+\dddd{h^\aaaa}{\GGGG(\aaaa)} \sum_{k=1}^{m-1} \dddd{y_{m-k}}{k^{1-\aaaa}}=F_m+O\llll( h^\aaaa \rrrr).
\end{equation}
Let $u_m$ be an approximation for the exact solution $y_m=y(x_m)$ of equations \eqref{16_2} and \eqref{17_1}. The numbers $u_m$ are computed with 
\begin{equation}\label{17_2}
u_m=F_m-\dddd{h^\aaaa}{\GGGG(\aaaa)} \sum_{k=1}^{m-1} \dddd{u_{m-k}}{k^{1-\aaaa}}, \quad u_0=0.
\end{equation}
In Theorem 6 and Theorem 11 we prove that numerical solution \eqref{17_2}   converges to the exact solution of equation \eqref{17_1} with accuracy $O\llll(h^\aaaa\rrrr)$.\\

\noindent
{\bf Example:} The fractional relaxation-oscillation equation 
$$y^{(\aaaa)}(x)+y(x)=x^p+\dddd{\GGGG(p+1)}{\GGGG(p-\aaaa+1)}x^{p-\aaaa},$$
has the solution $y(x)=x^p$. The equation is transformed to the fractional integral equation
\begin{equation}\label{17_3}
y(x)+I^\aaaa y(x)=x^p+\dddd{\GGGG(p+1)}{\GGGG(p+\aaaa+1)}x^{p+\aaaa}.
\end{equation}
The numerical results for the maximum error and order of numerical solution \eqref{17_2} for equation \eqref{17_3} are given in Table 2 and Table 3.
\setlength{\tabcolsep}{0.5em}
{ \renewcommand{\arraystretch}{1.1}
\begin{table}[ht]
	\caption{Maximum error and order of numerical solution \eqref{17_2} for  equation \eqref{17_3} with $p=1.05$  and $\aaaa=0.25,\aaaa=0.5,\aaaa=0.75$.}
	\small
	\centering
  \begin{tabular}{ l | c  c | c  c | c  c }
		\hline
		\hline
		\multirow{2}*{ $\quad \boldsymbol{h}$}  & \multicolumn{2}{c|}{$\boldsymbol{\aaaa=0.25}$} & \multicolumn{2}{c|}{$\boldsymbol{\aaaa=0.5}$}  & \multicolumn{2}{c}{$\boldsymbol{\aaaa=0.75}$} \\
		\cline{2-7}  
   & $Error$ & $Order$  & $Error$ & $Order$  & $Error$ & $Order$ \\ 
		\hline \hline
$0.003125$    & $0.1344240$  & $0.2863$  & $0.0264388$   & $0.5148$    & $0.00525169$  & $0.7544$       \\ 
$0.0015625$   & $0.0915092$  & $0.2799$  & $0.0185593$   & $0.5105$    & $0.00311695$  & $0.7526$       \\ 
$0.00078125$  & $0.0915092$  & $0.2748$  & $0.0130559$   & $0.5074$    & $0.00185130$  & $0.7516$        \\ 
$0.000390625$ & $0.0758594$  & $0.2706$  & $0.0091983$   & $0.5053$    & $0.00110006$  & $0.7510$        \\
\hline
  \end{tabular}
	\end{table}
	}
	\setlength{\tabcolsep}{0.5em}
{ \renewcommand{\arraystretch}{1.1}
\begin{table}[ht]
	\caption{Maximum error and order of numerical solution \eqref{17_2} for equation \eqref{17_3} with $p=1.05$  and $\aaaa=1.25,\aaaa=1.5,\aaaa=1.75$.}
	\small
	\centering
  \begin{tabular}{ l | c  c | c  c | c  c }
		\hline
		\hline
		\multirow{2}*{ $\quad \boldsymbol{h}$}  & \multicolumn{2}{c|}{$\boldsymbol{\aaaa=0.25}$} & \multicolumn{2}{c|}{$\boldsymbol{\aaaa=0.5}$}  & \multicolumn{2}{c}{$\boldsymbol{\aaaa=0.75}$} \\
		\cline{2-7}  
   & $Error$ & $Order$  & $Error$ & $Order$  & $Error$ & $Order$ \\ 
		\hline \hline 
$0.003125$    & $0.00018059$  & $1.2517$  & $0.00003093$   & $1.5082$    & $5.3\times 10^{-6}$  & $1.7781$       \\ 
$0.0015625$   & $0.00007588$  & $1.2509$  & $0.00001089$   & $1.5056$    & $1.5\times 10^{-6}$  & $1.7733$       \\ 
$0.00078125$  & $0.00003189$  & $1.2505$  & $3.8\times 10^{-6}$   & $1.5039$    & $4.5\times 10^{-7}$  & $1.7692$        \\ 
$0.000390625$ & $0.00001341$  & $1.2503$  & $1.4\times 10^{-6}$   & $1.5027$    & $1.3\times 10^{-7}$  & $1.7658$        \\
\hline
  \end{tabular}
	\end{table}
	}
\subsection{Stability and Convergence}
Let $e_m=y_m-u_m$ be the error of numerical solution \eqref{17_2} at the point $x=x_m$.
From \eqref{17_15} and \eqref{17_2} the error $e_m$  satisfies 
$$e_m=-\dddd{h^\aaaa}{\GGGG(\aaaa)}\sum_{k=1}^{m-1} k^{\aaaa-1}e_{m-k}+a_m h^\aaaa,\quad (e_0=0),$$
where $a_m h^\aaaa$ is the  error of approximation \eqref{16_15} at the point $x=x_m$. When the solution of the fractional-relaxation equation \eqref{16_2} and the corresponding integral equation \eqref{17_1} is a bounded function the sequence of numbers $a_m$ is bounded. Let $A$ be a positive constant, such that $|a_m|<A$.
\begin{clm} Let $0<\aaaa<2$ and $m$ be a positive integer. Then
$$1+2^{\aaaa-1}+\cdots+(m-1)^{\aaaa-1}<\dddd{m^\aaaa}{\aaaa}.
$$
\end{clm}
\begin{proof} The function $x^{\aaaa-1}$ is increasing when $1<\aaaa<2$ and decreasing for $0<\aaaa<1$. Then
$$1+2^{\aaaa-1}+\cdots+(m-1)^{\aaaa-1}<\int_0^m x^{\aaaa-1} dx=\llll.\dddd{x^\aaaa}{\aaaa}\rrrr|_0^m=\dddd{m^\aaaa}{\aaaa}.$$
\end{proof}
\begin{thm} Let $1<\aaaa<2$. Then the error of numerical solution \eqref{17_2} for equation \eqref{17_1} satisfies
\begin{equation}\label{19_1}
|e_m|<\llll(\dddd{\GGGG(\aaaa+1)A }{\GGGG(\aaaa+1)-1}\rrrr)h^\aaaa.
\end{equation}
\end{thm}
\begin{proof} Induction on $m$. When $m=1$
$$|e_1|=|a_1|h^\aaaa<Ah^\aaaa<\dddd{\GGGG(\aaaa+1)A h^\aaaa}{\GGGG(\aaaa+1)-1}.$$
Suppose that \eqref{19_1} holds for $k=1,\cdots,m-1$. From the triangle inequality
$$|e_m|\leq \dddd{h^\aaaa}{\GGGG(\aaaa)}\sum_{k=1}^{m-1} k^{\aaaa-1}|e_{m-k}|+|a_m|h^\aaaa.$$
From the induction assumption and Claim 5
$$|e_m|\leq \dddd{h^\aaaa}{\GGGG(\aaaa)}\dddd{\GGGG(\aaaa+1)A h^\aaaa}{\GGGG(\aaaa+1)-1}\sum_{k=1}^{m-1} k^{\aaaa-1}+Ah^\aaaa<\dddd{(mh)^\aaaa}{\GGGG(\aaaa+1)}\dddd{\GGGG(\aaaa+1)A h^\aaaa}{\GGGG(\aaaa+1)-1}+Ah^\aaaa,$$
$$|e_m|<\dddd{A h^\aaaa}{\GGGG(\aaaa+1)-1}+Ah^\aaaa=\dddd{\GGGG(\aaaa+1)A h^\aaaa}{\GGGG(\aaaa+1)-1}.$$
\end{proof}
The functions $2^x$ and $\GGGG(1+x)$ are increasing for $x\in [1,2]$.  The gamma function satisfies $\GGGG(1+x)<1$ when $0<x<1$.
\begin{clm}
The function $g(x)=2^x\GGGG(1+x)$ is increasing on the interval $[0,1]$.
\end{clm}
\begin{proof} The function $g$ has values $g(0)=1$ and $g(1)=2$. Let
$$l(x)=\ln g(x)=x \ln 2+\ln\GGGG(1+x).$$
The function $l(x)$ has first derivative
$$l'(x)=\ln 2+\Psi(1+x).$$
From the expansion formula for the digamma function
$$l'(x)=\ln 2-\gggg-\sum_{n=1}^\infty \zzzz(k+1)(-x)^k,$$
$$l'(x)=\ln 2-\gggg+\sum_{n=1}^\infty x^{2k-1} (\zzzz(2k)-\zzzz(2k+1)x).$$
The zeta function is decreasing for $x>1$. Then
$$\zzzz(2k)-\zzzz(2k+1)x>\zzzz(2k)-\zzzz(2k+1)>0.$$
Hence
$$l'(x)=\ln 2-\gggg+\sum_{n=1}^\infty x^{2k-1} (\zzzz(2k)-\zzzz(2k+1)x)>\ln 2-\gggg>0.1>0.$$
The functions $l(x)$ and $g(x)=e^{l(x)}$ are increasing on the interval $[0,1]$.
\end{proof}
\begin{clm}  Let $0<\aaaa<1$ and $m<n$ be positive integers. Then
$$m^{\aaaa-1}+\cdots+n^{\aaaa-1}<\dddd{n^\aaaa-(m-1)^\aaaa}{\aaaa}.$$
\end{clm}
\begin{proof} The function $x^{\aaaa-1}$ is decreasing. Then
$$m^{\aaaa-1}+\cdots+n^{\aaaa-1}<\int_{m-1}^n x^{\aaaa-1} dx=\llll.\dddd{x^\aaaa}{\aaaa}\rrrr|_{m-1}^n=\dddd{n^\aaaa-(m-1)^\aaaa}{\aaaa}.$$
\end{proof}
 Let $M=\left\lceil \frac{n}{2}\rrrr\rceil$ and
$$C_0=\dddd{2^\aaaa \GGGG(1+\aaaa)A}{2^\aaaa \GGGG(1+\aaaa)-1}.$$
From Claim 7, the number $C_0$ is positive.
\begin{lem} Let $1\leq m\leq M$. Then 
\begin{equation}\label{21_1}
|e_m|<C_0 h^\aaaa.
\end{equation}
\end{lem}
\begin{proof} Induction on $m$. When $m=1$
$$ |e_1|<|a_1|h^\aaaa<Ah^\aaaa<\dddd{2^\aaaa \GGGG(1+\aaaa)A h^\aaaa}{2^\aaaa \GGGG(1+\aaaa)-1}.
$$
Suppose that \eqref{21_1} holds for $k=1,\cdots,m-1$.
$$|e_m|\leq\dddd{h^\aaaa}{\GGGG(\aaaa)}\sum_{k=1}^{m-1} k^{\aaaa-1}|e_{m-k}|+|a_m|h^\aaaa,$$
$$|e_m|<\dddd{h^\aaaa}{\GGGG(\aaaa)}C_0 h^\aaaa\sum_{k=1}^{m-1} k^{\aaaa-1}+A h^\aaaa<\dddd{h^\aaaa}{\GGGG(\aaaa)}C_0 h^\aaaa \dddd{M^\aaaa}{\aaaa}+A h^\aaaa,$$
$$|e_m|<\dddd{h^\aaaa}{\GGGG(1+\aaaa)}C_0 h^\aaaa \llll(\dddd{n}{2}\rrrr)^\aaaa+A h^\aaaa=\dddd{C_0 h^\aaaa}{2^\aaaa\GGGG(1+\aaaa)}+A h^\aaaa,$$
$$|e_m|<\dddd{ A h^\aaaa}{2^\aaaa\GGGG(1+\aaaa)-1}+A h^\aaaa=\dddd{2^\aaaa\GGGG(1+\aaaa) A h^\aaaa}{2^\aaaa\GGGG(1+\aaaa)-1}.$$
\end{proof}
Denote
$$C_1=\dddd{2^{1-\aaaa}C_0+\GGGG(\aaaa)A}{\GGGG(\aaaa)},\quad 
C_2=\dddd{C_0+2^\aaaa \GGGG(1+\aaaa)C_1}{2^\aaaa \GGGG(1+\aaaa)-1}.$$
The number $C_2$ is greater than $Max\{C_0,C_1\}$.
\begin{cor} The error $e_M$ satisfies
$$\dddd{h^\aaaa}{\GGGG(\aaaa)} (n-M)^{\aaaa-1}|e_M|<(C_1-A) h^\aaaa.
$$
\end{cor}
\begin{proof} 
\begin{align*}
\dddd{h^\aaaa}{\GGGG(\aaaa)} (n-M)^{\aaaa-1}&|e_M|+Ah^\aaaa<\dddd{h^\aaaa}{\GGGG(\aaaa)} \llll(\dddd{n}{2}\rrrr)^{\aaaa-1}C_0 h^\aaaa+Ah^\aaaa\\
&=\dddd{2^{1-\aaaa}C_0h^\aaaa}{\GGGG(\aaaa)n} n^\aaaa  h^\aaaa+Ah^\aaaa<\dddd{2^{1-\aaaa}C_0h^\aaaa}{\GGGG(\aaaa)}+Ah^\aaaa=C_1 h^\aaaa.
\end{align*}
\end{proof}
\begin{thm} Let $0<\aaaa<1$ and $M\leq m\leq n$. Then
\begin{equation}\label{22_1}
|e_m|<C_2 h^\aaaa.
\end{equation}
\end{thm}
\begin{proof} Induction on $m$. Suppose that \eqref{22_1}  holds for $k=M,\cdots,m-1$. 
$$|e_m|\leq \dddd{h^\aaaa}{\GGGG(\aaaa)}\sum_{k=1}^{m-1} k^{\aaaa-1}|e_{m-k}|+|a_m|h^\aaaa,$$
\begin{align*}
|e_m|<\dddd{h^\aaaa}{\GGGG(\aaaa)}\sum_{k=1}^{m-M-1} k^{\aaaa-1}|e_{m-k}|+&\dddd{h^\aaaa}{\GGGG(\aaaa)}\sum_{k=m-M+1}^{m-1} k^{\aaaa-1}|e_{n-k}|\\
&+\dddd{h^\aaaa}{\GGGG(\aaaa)}(m-M)^{\aaaa-1}|e_M|+A h^\aaaa.
\end{align*}
From the induction assumption and Corollary 10
\begin{align*}
|e_m|<\dddd{h^\aaaa}{\GGGG(\aaaa)}C_2 h^\aaaa\sum_{k=1}^{m-M-1} k^{\aaaa-1}+\dddd{h^\aaaa}{\GGGG(\aaaa)}C_0 h^\aaaa \sum_{k=m-M+1}^{m-1} k^{\aaaa-1}+C_1 h^\aaaa.
\end{align*}
From Claim 5 and Claim 8
\begin{align*}
|e_m|<\dddd{h^\aaaa}{\GGGG(\aaaa)}C_2 h^\aaaa \dddd{(m-M)^{\aaaa}}{\aaaa}+\dddd{h^\aaaa}{\GGGG(\aaaa)}C_0 h^\aaaa \dddd{(m-1)^{\aaaa}-(m-M)^{\aaaa}}{\aaaa}+C_1 h^\aaaa.
\end{align*}
The function $(x-1)^\aaaa-(x-M)^\aaaa$ is continuous and decreasing  for $x\geq M$. Its maximum $(M-1)^\aaaa$ is attained at $x=M$. Hence
\begin{align*} 
|e_m|<\dddd{h^\aaaa}{\GGGG(1+\aaaa)}C_2 h^\aaaa \llll(\dddd{n-M}{2}\rrrr)^{\aaaa}+\dddd{h^\aaaa}{\GGGG(1+\aaaa)}C_0 h^\aaaa (M-1)^{\aaaa}+C_1 h^\aaaa,
\end{align*}
\begin{align*} 
|e_m|<\dddd{h^\aaaa}{\GGGG(1+\aaaa)}C_2 h^\aaaa \llll(\dddd{n}{2}\rrrr)^{\aaaa}+\dddd{h^\aaaa}{\GGGG(1+\aaaa)}C_0 h^\aaaa \llll(\dddd{n}{2}\rrrr)^{\aaaa}+C_1 h^\aaaa,
\end{align*}
\begin{align*} 
|e_m|<\dddd{C_2h^\aaaa}{2^\aaaa \GGGG(1+\aaaa)} +\dddd{C_0 h^\aaaa}{2^\aaaa \GGGG(1+\aaaa)} +C_1 h^\aaaa=C_2 h^\aaaa .
\end{align*}
\end{proof}
\section{Higher-Order Numerical Solutions of the  Fractional Relaxation-Oscillation Equation}
In this section we derive approximations for the fractional integral $I^\aaaa y(x)$ of order $1+\aaaa,2+\aaaa,3+\aaaa$ and $4+\aaaa$ when the function $y$ satisfies the condition
\begin{equation}\label{23_1}
y(0)=y'(0)=y''(0)=y'''(0)=0.
\end{equation}
Denote by Equation I, equation \eqref{17_3} with $p=4$.
\begin{align}\tag{{\bf Equation\;1}}
y(x)+I^\aaaa y(x)=x^4+\dddd{24}{\GGGG(5+\aaaa)}x^{4+\aaaa},\quad y(0)=y'(0)=0.
\end{align}
The solution $y(x)=x^4$ of Equation 1 satisfies \eqref{23_1}. The Miller-Ross sequential  derivative  is defined as
$$y^{[\aaaa_1]}(x)=y^{(\aaaa_1)}(x),\quad y^{[\aaaa_1+\aaaa_2+\cdots+\aaaa_n]}(x)=D^{\aaaa_1}D^{\aaaa_2}\cdots D^{\aaaa_n}y(x).$$
Denote
 $$y^{[n\aaaa]}(x)=y^{[\aaaa+\aaaa+\cdots+\aaaa]}(x)=D^{\aaaa}D^{\aaaa}\cdots D^{\aaaa}y(x).$$

\begin{clm} Let $y$ be a sufficiently differentiable function and $0<\aaaa<1$. 
$$ y^{[1+(1-\aaaa)]}(x)=y^{(2-\aaaa)}(x)+\dddd{y'(0)}{\GGGG(\aaaa)}\dddd{1}{x^{1-\aaaa}}.$$
\end{clm}
\begin{proof} 
\begin{align*} 
\dddd{d}{d x} y^{(1-\aaaa)}(x)&=\dddd{1}{\GGGG(\aaaa)}\dddd{d}{dx}\int_0^x \dddd{y'(\xi)}{(x-\xi)^{1-\aaaa}}d\xi=\dddd{1}{\GGGG (\aaaa)}\dddd{d}{d x}\int_0^x y'(\xi)d\llll( -\dddd{(x-\xi)^\aaaa}{\aaaa}\rrrr)\\
&=\dddd{1}{\GGGG (1+\aaaa)}\dddd{d}{d x}\llll(\llll. -y'(\xi)(x-\xi)^\aaaa\rrrr|_0^x+\int_0^x (x-\xi)^\aaaa d y'(\xi)\rrrr)\\
&=\dddd{1}{\GGGG (1+\aaaa)}\llll(\dddd{d}{d x}y'(0)x^\aaaa+\dddd{d}{d x}\int_0^x (x-\xi)^\aaaa y''(\xi)d\xi\rrrr)\\
&=\dddd{1}{\GGGG (1+\aaaa)}\llll(\aaaa y'(0)x^{\aaaa-1}+\aaaa\int_0^x (x-\xi)^{\aaaa-1} y''(\xi)d\xi\rrrr)\\
&=\dddd{y'(0)}{\GGGG (\aaaa) x^{1-\aaaa}}+\dddd{1}{\GGGG (\aaaa) }\int_0^x  \dddd{y''(\xi)}{(x-\xi)^{1-\aaaa}}d\xi=\dddd{y'(0)}{\GGGG (\aaaa) x^{1-\aaaa}}+y^{(2-\aaaa)}(x).
\end{align*}
\end{proof}
The fractional relaxation equation
\begin{equation}\label{24_1}
y^{(\aaaa)}(x)+y(x)=e^x+x^{1-\aaaa}E_{1,2-\aaaa}\llll(x\rrrr),\quad y(0)=1,
\end{equation}
 has the solution $y(x)=e^x$. Now we determine the values of the derivatives $y'(0),y''(0),y'''(0)$ of the solution of equation \eqref{24_1}. By applying fractional differentiation of order $1-\aaaa$  we obtain
\begin{equation*}
D^{1-\aaaa}\llll(y^{(\aaaa)}(x)+y(x)\rrrr)=D^{1-\aaaa}\llll(e^x+x^{1-\aaaa}E_{1,2-\aaaa}\llll(x\rrrr)\rrrr),
\end{equation*}
\begin{equation}\label{24_2}
y'(x)+y^{(1-\aaaa)}(x)=e^x+x^{\aaaa}E_{1,1+\aaaa}\llll(x\rrrr).
\end{equation}
By setting $x=0$ we obtain  $y'(0)=1$. Differentiate equation  \eqref{24_2} 
\begin{equation*}
\dddd{d}{dx}\llll(y'(x)+y^{(1-\aaaa)}(x)\rrrr)=\dddd{d}{dx}\llll(e^x+\sum_{k=0}^\infty\dddd{x^{k+\aaaa}}{\GGGG(k+\aaaa+1)}\rrrr).
\end{equation*}
From Claim 12,
\begin{equation*}
y''(x)+y^{(2-\aaaa)}(x)+\dddd{1}{\GGGG(\aaaa)x^{1-\aaaa}}=e^x+\sum_{k=0}^\infty\dddd{x^{k+\aaaa-1}}{\GGGG(k+\aaaa)},
\end{equation*}
\begin{equation*}
y''(x)+y^{(2-\aaaa)}(x)=e^x+\sum_{k=1}^\infty\dddd{x^{k+\aaaa-1}}{\GGGG(k+\aaaa)},
\end{equation*}
\begin{equation}\label{24_3}
y''(x)+y^{(2-\aaaa)}(x)=e^x+x^\aaaa E_{1,1+\aaaa}(x).
\end{equation}
Similarly, by differentiating \eqref{24_3}
\begin{equation}\label{24_4}
y'''(x)+y^{(3-\aaaa)}(x)=e^x+x^\aaaa E_{1,1+\aaaa}(x).
\end{equation}
By setting $x=0$ in \eqref{24_3} and \eqref{24_4} we obtain $y''(0)=y'''(0)=1$. Let
$$z(x)=y(x)-\sum_{k=0}^m \dddd{x^m}{m!}.$$
The function $z(x)$ satisfies  $z(0)=z'(0)=0$ for  $m=0,1,2$ and condition \eqref{23_1} when $m=3$.
$$z^{(\aaaa)}(x)=y^{(\aaaa)}(x)-\sum_{k=0}^m \dddd{x^{m-\aaaa}}{\GGGG(m+1-\aaaa)}.$$
The function $z(x)$ is a solution of the fractional relaxation equation
$$z^{(\aaaa)}(x)+z(x)=e^x+x^{1-\aaaa}E_{1,2-\aaaa}\llll(x\rrrr)-\sum_{k=0}^m \dddd{x^m}{m!}-\sum_{k=1}^m \dddd{x^{m-\aaaa}}{\GGGG(m+1-\aaaa)}.
$$
By applying fractional integration of order $\aaaa$ we obtain
$$I^\aaaa\llll(z^{(\aaaa)}(x)+z(x)\rrrr)=I^\aaaa\llll(e^x+x^{1-\aaaa}E_{1,2-\aaaa}\llll(x\rrrr)-\sum_{k=0}^m \dddd{x^m}{m!}-\sum_{k=1}^m \dddd{x^{m-\aaaa}}{\GGGG(m+1-\aaaa)}\rrrr),$$
$$z(x)+I^\aaaa z(x)=e^x-1+x^{\aaaa}E_{1,1+\aaaa}\llll(x\rrrr)-\sum_{k=0}^m \dddd{x^{m+\aaaa}}{\GGGG(m+1+\aaaa)}-\sum_{k=1}^m \dddd{x^{m}}{m!}.$$
Denote  by Equation 2[m] the integral equation
\begin{align}\tag{{\bf Equation\;2[m]}}
z(x)+I^\aaaa z(x)=e^x+x^{\aaaa}E&_{1,1+\aaaa}\llll(x\rrrr)\\
&-\sum_{k=0}^m \llll(\dddd{x^{m+\aaaa}}{\GGGG(m+1+\aaaa)}+ \dddd{x^{m}}{m!}\rrrr).\nonumber
\end{align}
Equation 2[m] has the solution 
$$z(x)=e^x-\sum_{k=0}^m \dddd{x^m}{m!}.$$
The fractional Taylor polynomials of degree $m$ of the function $y$ at the point $x=0$ are defined as
$$T^{(\aaaa)}_m(x)= \sum_{n=0}^m \dddd{y^{[n\aaaa]}(0)}{\GGGG(\aaaa n+1)}x^{\aaaa n}.$$
The fractional Taylor polynomials approximate the value of the solution of the fractional relaxation-oscillation equation  $y(h)\approx T^{(\aaaa)}_m (h)$, when $h$ is a small positive number. The fractional relaxation equation
\begin{equation}\label{25_1}
y^{(\aaaa)}(x)+y(x)=x^{2\aaaa},\quad y(0)=1,
\end{equation}
 has the solution $y(x)=E_{\aaaa}\llll(-x^{a}\rrrr)+\GGGG(1+2\aaaa) x^{3\aaaa}E_{\aaaa,1+3\aaaa}\llll(-x^{a}\rrrr)$. The first derivative of the solution is undefined at $x=0$. We determine the Miller-Ross derivatives of the solution at $x=0$.
$$y^{[\aaaa]}(0)=y^{(\aaaa)}(0)=-1.$$
By applying fractional differentiation of order $\aaaa$ to equation \eqref{25_1}  we obtain
$$y^{[2\aaaa]}(x)+y^{[\aaaa]}(x)=\dddd{\GGGG(1+2\aaaa)}{\GGGG(1+\aaaa)}x^{\aaaa},\qquad y^{[2\aaaa]}(0)=-1,$$
$$y^{[3\aaaa]}(x)+y^{[2\aaaa]}(x)=\GGGG(1+2\aaaa),\qquad y^{[3\aaaa]}(0)=\GGGG(1+2\aaaa)-1,$$
$$y^{[4\aaaa]}(x)+y^{[3\aaaa]}(x)=0,\qquad\qquad\qquad y^{[4\aaaa]}(0)=1-\GGGG(1+2\aaaa).$$
 We can show by induction that
$$y^{[n\aaaa]}(0)=(-1)^n (1-\GGGG(1+2\aaaa)).$$
The solution of equation \eqref{25_1} has fractional Taylor polynomials
$$T^{(\aaaa)}_m(x)= 1-\dddd{x^\aaaa}{\GGGG(1+\aaaa)}+\dddd{x^{2\aaaa}}{\GGGG(1+2\aaaa)}+(1-\GGGG(1+2\aaaa))\sum_{k=3}^m \dddd{(-1)^k x^{k\aaaa}}{\GGGG(1+k\aaaa)}.$$
Let $z(x)=y(x)-T^{(\aaaa)}_m(x)$,
$$z(x)=y(x)-1+\dddd{x^\aaaa}{\GGGG(1+\aaaa)}-\dddd{x^{2\aaaa}}{\GGGG(1+2\aaaa)}+(\GGGG(1+2\aaaa)-1)\sum_{k=3}^m \dddd{(-1)^k x^{k\aaaa}}{\GGGG(1+k\aaaa)}.$$
When $m>\left\lceil 3/\aaaa\right\rceil$ the function $z(x)$ satisfies condition \eqref{23_1}.
\begin{align*}
z^{(\aaaa)}(x)=y^{(\aaaa)}(x)+1-\dddd{x^\aaaa}{\GGGG(1+\aaaa)}-&\dddd{(\GGGG(1+2\aaaa)-1)x^{2\aaaa}}{\GGGG(1+2\aaaa)}\\
&+(\GGGG(1+2\aaaa)-1)\sum_{k=4}^m \dddd{(-1)^k x^{(k-1)\aaaa}}{\GGGG(1+(k-1)\aaaa)}.
\end{align*}
The function $z(x)$ is a solution of the fractional relaxation equation
$$z^{(\aaaa)}(x)+z(x)=(-1)^m \dddd{\GGGG(1+2\aaaa)-1}{\GGGG(1+m\aaaa)}x^{m\aaaa}.$$
By applying fractional integration of order $\aaaa$ we obtain
$$I^\aaaa\llll(z^{(\aaaa)}(x)+z(x)\rrrr)=(-1)^m \dddd{\GGGG(1+2\aaaa)-1}{\GGGG(1+m\aaaa)}I^\aaaa(x^{m\aaaa}),$$
\begin{align}\tag{{\bf Equation\;3[m]}}
z(x)+I^\aaaa z(x)=(-1)^m \dddd{\GGGG(1+2\aaaa)-1}{\GGGG(1+(m+1)\aaaa)}x^{(m+1)\aaaa}.
\end{align}
 Equation3[m] has the solution 
\begin{align*}
z(x)=E_{\aaaa}\llll(-x^{a}\rrrr)+\GGGG(1+&2\aaaa) x^{3\aaaa}E_{\aaaa,1+3\aaaa}\llll(-x^{a}\rrrr)-1+\dddd{x^\aaaa}{\GGGG(1+\aaaa)}\\
&-\dddd{x^{2\aaaa}}{\GGGG(1+2\aaaa)}+(\GGGG(1+2\aaaa)-1)\sum_{k=3}^m \dddd{(-1)^k x^{k\aaaa}}{\GGGG(1+k\aaaa)}.
\end{align*}
\subsection{Numerical solution of order $\boldsymbol{1+\aaaa}$} 
The three integral equations Equation 1, Equation 2[m] and Equation 3[m] have the form
\begin{equation}\label{27_0}
y(x)+I^\aaaa y(x)=F(x).
\end{equation}
We can ensure that sufficient number of derivatives of the solutions of Equation 2[m] and Equation 3[m] at the initial point $x=0$ are equal to zero by choosing appropriate values of the parameter $m$. In Corollary 4 we obtained the approximation for the fractional integral
\begin{align}\label{27_1}
h^{\aaaa}\sum_{k=1}^{n-1}& \dddd{ y(x-kh)}{k^{1-\aaaa}}=\GGGG(\aaaa)I^{\aaaa}y(x)+\zzzz(1-\aaaa) y(x)h^{\aaaa}-\zzzz(-\aaaa)y'(x)h^{1+\aaaa}\nonumber\\
+&\zzzz(-1-\aaaa)\dddd{y''(x)}{2}h^{2+\aaaa}
-\zzzz(-2-\aaaa)\dddd{y'''(x)}{6}h^{3+\aaaa}+O\llll(h^{4+\aaaa} \rrrr).
\end{align}
when the function $y$ satisfies  \eqref{23_1}. Let $x=x_n=n h$ and $y_n=y(x_n)$. 
\begin{cor} Suppose that $y(0)=y'(0)=0$. Then
\begin{align}\label{27_2}
\dddd{h^\aaaa}{\GGGG(\aaaa)}\llll(-\zzzz(1-\aaaa)y_n+\sum_{k=1}^{n-1}\dddd{y_{n-k}}{k^{1-\aaaa}}   \rrrr)=I^\aaaa y_n+O\llll(h^{1+\aaaa}\rrrr).
\end{align}
\end{cor}
By approximating the fractional integral  with \eqref{27_2} we obtain the numerical solution of equation \eqref{27_0}
\begin{equation}\label{27_3}
u_n=\dddd{1}{\GGGG(\aaaa)-\zzzz(1-\aaaa)h^\aaaa}\llll(\GGGG(\aaaa)F_n-h^\aaaa  \sum_{k=1}^{n-1} \dddd{u_{n-k}}{k^{1-\aaaa}}\rrrr),\quad u_0=0.
\end{equation}
	In Figure 1 and Figure 2 we compare numerical solutions \eqref{17_2} and \eqref{27_3} for Equation 1 when $\aaaa=0.5$ and $\aaaa=1.05$.
The numerical results for the maximum error and the order of numerical solution \eqref{27_3} for Equation 1 with $\aaaa=0.25,1.25$, Equation 2[1] with $\aaaa=0.5,1.5$ and Equation 3[2] with $\aaaa=0.75,1.75$ are given in Table 4 and Table 5.
\setlength{\tabcolsep}{0.5em}
{ \renewcommand{\arraystretch}{1.1}
\begin{table}[ht]
	\caption{Maximum error and order of numerical solution \eqref{27_3} for $\aaaa=0.25$, $\aaaa=0.5$ and $\aaaa=0.75$.}
	\small
	\centering
  \begin{tabular}{ l | l  l | l  l | l  l }
		\hline
		\hline
		\multirow{2}*{ $\quad \boldsymbol{h}$}  & \multicolumn{2}{c|}{{\bf Equation\;1}} & \multicolumn{2}{c|}{{\bf Equation\;2[1]}}  & \multicolumn{2}{c}{{\bf Equation\;3[2]}} \\
		\cline{2-7}  
   & $Error$ & $Order$  & $Error$ & $Order$  & $Error$ & $Order$ \\ 
		\hline \hline
$0.003125$    & $0.00015093$  & $1.2490$  & $0.00002066$          & $1.5000$    & $4.8\times 10^{-7}$  & $1.7500$       \\ 
$0.0015625$   & $0.00006348$  & $1.2490$  & $7.3\times 10^{-6}$    & $1.5000$   & $1.4\times 10^{-7}$  & $1.7500$       \\ 
$0.00078125$  & $0.00002670$  & $1.2500$  & $2.6\times 10^{-6}$   & $1.5000$    & $4.2\times 10^{-8}$  & $1.7500$        \\ 
$0.000390625$ & $0.00001123$  & $1.2500$  & $9.1\times 10^{-7}$   & $1.5000$    & $1.3\times 10^{-8}$  & $1.7500$        \\
\hline
  \end{tabular}
	\end{table}
	}
 \setlength{\tabcolsep}{0.5em}
  { \renewcommand{\arraystretch}{1.1}
\begin{table}[ht]
	\caption{Maximum error and order of numerical solution \eqref{27_3} for $\aaaa=0.1.25$, $\aaaa=1.5$ and $\aaaa=1.75$.}
	\small
	\centering
  \begin{tabular}{ l | l  l | l  l | l  l }
		\hline
		\hline
		\multirow{2}*{ $\quad \boldsymbol{h}$}  & \multicolumn{2}{c|}{{\bf Equation\;1}} & \multicolumn{2}{c|}{{\bf Equation\;2[1]}}  & \multicolumn{2}{c}{{\bf Equation\;3[2]}} \\
		\cline{2-7}  
   & $Error$ & $Order$  & $Error$ & $Order$  & $Error$ & $Order$ \\ 
		\hline \hline
$0.003125$    & $4.2\times 10^{-7}$  & $2.2510$  & $2.1\times 10^{-8}$    & $2.5010$    & $3.8\times 10^{-10}$  & $2.7580$   \\ 
$0.0015625$   & $8.9\times 10^{-8}$  & $2.2500$  & $3.8\times 10^{-9}$    & $2.5000$    & $5.7\times 10^{-11}$  & $2.7540$   \\ 
$0.00078125$  & $1.9\times 10^{-8}$  & $2.2500$  & $6.7\times 10^{-10}$   & $2.5000$    & $8.4\times 10^{-12}$  & $2.7520$   \\ 
$0.000390625$ & $3.9\times 10^{-9}$  & $2.2500$  & $1.2\times 10^{-10}$   & $2.5000$    & $1.3\times 10^{-12}$  & $2.7510$   \\
\hline
  \end{tabular}
	\end{table}
	}
	\begin{figure}[t!]
\centering
\begin{minipage}{.46\textwidth}
  \centering
  \includegraphics[width=.95\linewidth]{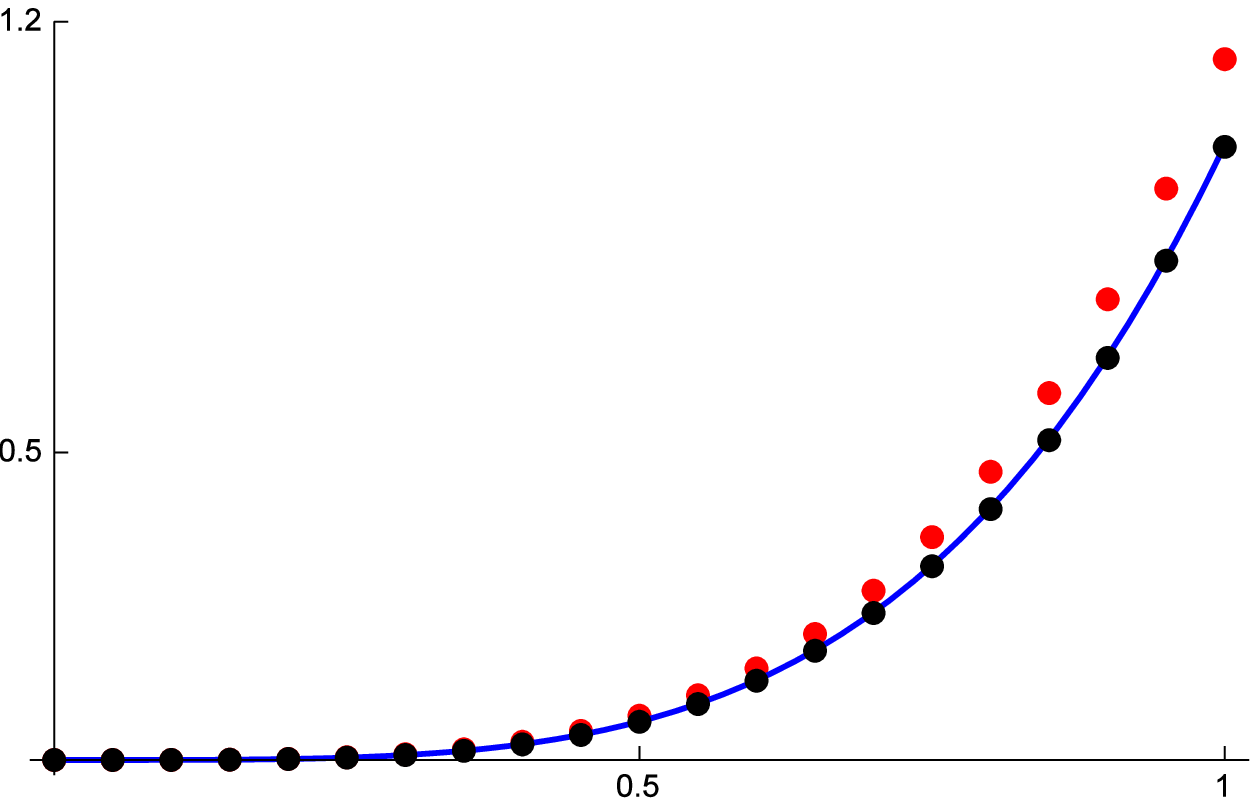}
  \captionof{figure}{Graph of the exact solution of Equation 1 and  numerical solutions \eqref{17_2}-red and \eqref{27_3}-black for $h=0.05,\alpha=0.5$.}
  \label{fig:test1}
\end{minipage}%
\hspace{0.5cm}
\begin{minipage}{.46\textwidth}
  \centering
  \includegraphics[width=.95\linewidth]{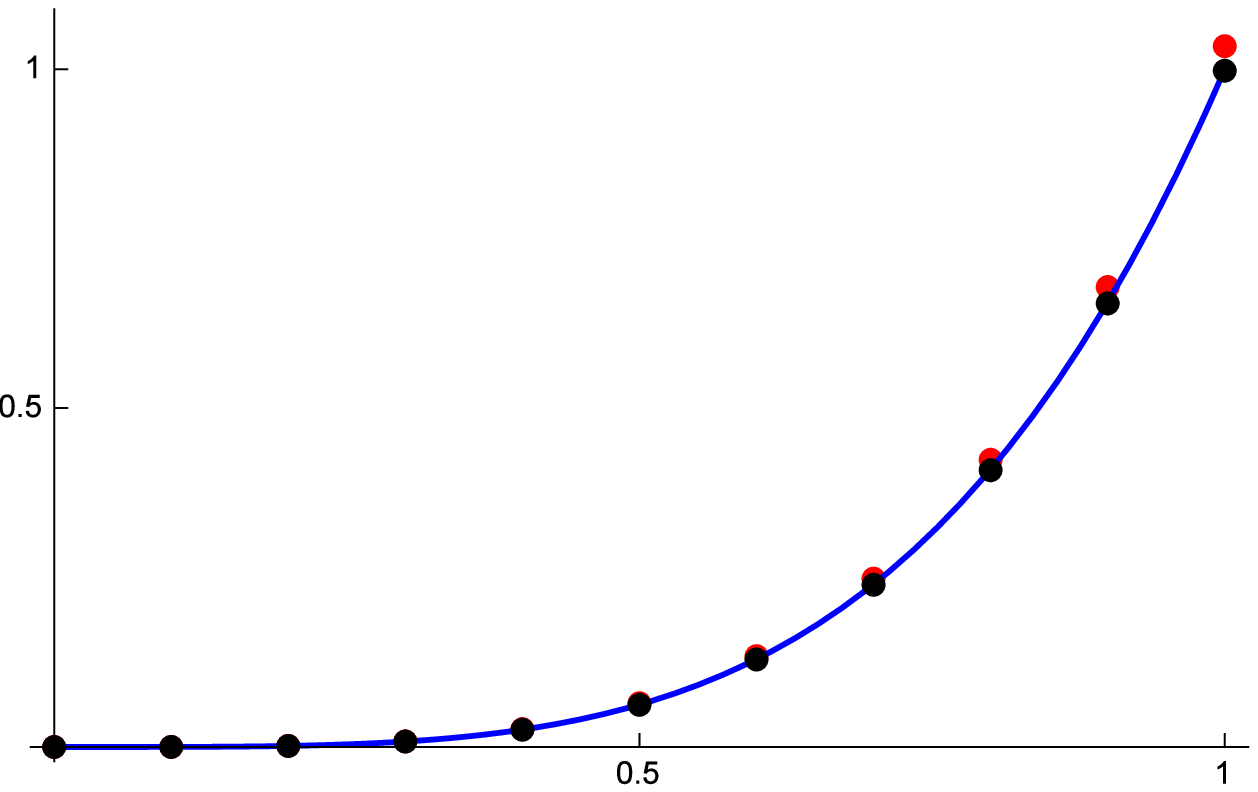}
  \captionof{figure}{Graph of the exact solution of Equation 1 and  numerical solutions  \eqref{17_2}-red and \eqref{27_3}-black for $h=0.1,\alpha=1.05$.}
  \label{fig:test2}
\end{minipage}
\end{figure}
\subsection{Numerical solution of order $\boldsymbol{2+\aaaa}$} 
By approximating the derivative $y'_n$  in approximation \eqref{27_1} using a first-order backward difference approximation 
$$y'_n=\dddd{y_n-y_{n-1}}{h}+O(h),$$
we obtain an approximation \eqref{29_1} for the fractional integral of order $2+\aaaa$.
\begin{clm} Suppose that $y(0)=y'(0)=0$. Then
\begin{align}\label{29_1}
\dddd{h^\aaaa}{\GGGG(\aaaa)} \llll((\zzzz(-\aaaa)-\zzzz(1-\aaaa))y_n-\zzzz(-\aaaa)y_{n-1}+\sum_{k=1}^{n-1}\dddd{y_{n-k}}{k^{1-\aaaa}}   \rrrr)=I^\aaaa y_n+O\llll(h^{2+\aaaa}\rrrr).
\end{align}
\end{clm}
The numerical solution of equation \eqref{29_2} which uses approximation \eqref{29_1} for the fractional integral is computed with $u_0=0$ and
\begin{equation}\label{29_2}
u_n=\dddd{1}{\GGGG(\aaaa)+(\zzzz(-\aaaa)-\zzzz(1-\aaaa))h^\aaaa}\llll(\GGGG(\aaaa)F_n+\zzzz(-\aaaa)h^\aaaa u_{n-1}-h^\aaaa \sum_{k=1}^{n-1} \dddd{u_{n-k}}{k^{1-\aaaa}}\rrrr).
\end{equation}
The numerical results for the maximum error and the order of numerical solution \eqref{29_2} for Equation 1 with $\aaaa=0.3,1.3$, Equation 2[2] with $\aaaa=0.5,1.5$ and Equation 3[2] with $\aaaa=0.7,1.7$ are given in Table 6 and Table 7.
\setlength{\tabcolsep}{0.5em}
{ \renewcommand{\arraystretch}{1.1}
\begin{table}[h]
	\caption{Maximum error and order of numerical solution \eqref{29_2} for $\aaaa=0.3$, $\aaaa=0.5$ and $\aaaa=0.7$.}
	\small
	\centering
  \begin{tabular}{ l | l  l | l  l | l  l }
		\hline
		\hline
		\multirow{2}*{ $\quad \boldsymbol{h}$}  & \multicolumn{2}{c|}{{\bf Equation\;1}} & \multicolumn{2}{c|}{{\bf Equation\;2[2]}}  & \multicolumn{2}{c}{{\bf Equation\;3[2]}} \\
		\cline{2-7}  
   & $Error$ & $Order$  & $Error$ & $Order$  & $Error$ & $Order$ \\ 
		\hline \hline
$0.025$    & $0.00005799$         & $2.2727$  & $5.1\times 10^{-6}$   & $2.4796$    & $3.8\times 10^{-7}$   & $2.7141$       \\ 
$0.0125$   & $0.00001189$         & $2.2864$  & $9.0\times 10^{-7}$   & $2.4898$    & $5.8\times 10^{-8}$   & $2.7160$       \\ 
$0.00625$   & $2.4\times 10^{-6}$  & $2.2932$  & $1.6\times 10^{-7}$   & $2.4949$    & $8.8\times 10^{-9}$   & $2.7095$        \\ 
$0.003125$ & $4.9\times 10^{-7}$  & $2.2966$  & $2.8\times 10^{-8}$   & $2.4975$    & $1.4\times 10^{-9}$   & $2.7055$        \\
\hline
  \end{tabular}
	\end{table}
	}
 \setlength{\tabcolsep}{0.5em}
{ \renewcommand{\arraystretch}{1.1}
\begin{table}[h]
	\caption{Maximum error and order of numerical solution \eqref{29_2} for $\aaaa=1.3$, $\aaaa=1.5$ and $\aaaa=1.7$.}
	\small
	\centering
 \begin{tabular}{ l | l  l | l  l | l  l }
		\hline
		\hline
		\multirow{2}*{ $\quad \boldsymbol{h}$}  & \multicolumn{2}{c|}{{\bf Equation\;1}} & \multicolumn{2}{c|}{{\bf Equation\;2[2]}}  & \multicolumn{2}{c}{{\bf Equation\;3[2]}} \\
		\cline{2-7}  
   & $Error$ & $Order$  & $Error$ & $Order$  & $Error$ & $Order$ \\ 
		\hline \hline 
$0.025$    & $1.4\times 10^{-6}$  & $3.2791$  & $6.4\times 10^{-8}$    & $3.4899$   & $1.6\times 10^{-8}$  & $3.6762$       \\ 
$0.0125$   & $1.4\times 10^{-7}$  & $3.2896$  & $5.7\times 10^{-9}$    & $3.4957$   & $1.2\times 10^{-9}$  & $3.6881$       \\ 
$0.00625$   & $1.4\times 10^{-8}$  & $3.2948$  & $5.0\times 10^{-10}$   & $3.4984$   & $9.5\times 10^{-11}$ & $3.6941$        \\ 
$0.003125$ & $1.4\times 10^{-9}$  & $3.2974$  & $4.4\times 10^{-11}$   & $3.4995$   & $7.3\times 10^{-12}$ & $3.6970$        \\
\hline
  \end{tabular}
	\end{table}
	}
\subsection{ Numerical solution of order $\boldsymbol{3+\aaaa}$} 
By substituting the derivatives $y'_n$ and $y''_n$ in \eqref{27_1} using the approximations
$$y'_n=\dddd{1}{h}\llll(\dddd{3}{2}y_n-2y_{n-1}+\dddd{1}{2}y_{n-2}\rrrr)+O\llll(h^{2}\rrrr),$$
$$y''_n=\dddd{1}{h^2}\llll(y_n-2y_{n-1}+y_{n-2}\rrrr)+O\llll(h\rrrr),$$
we obtain an approximation \eqref{30_1} for the fractional integral with order $3+\aaaa$.
\begin{clm} Suppose that $y(0)=y'(0)=y''(0)=0$. Then
\begin{align}\label{30_1}
\dddd{h^\aaaa}{\GGGG(\aaaa)} \llll(c_0 y_n+c_1 y_{n-1}+c_2 y_{n-2}+\sum_{k=1}^{n-1}\dddd{y_{n-k}}{k^{1-\aaaa}}   \rrrr)=I^\aaaa y_n+O\llll(h^{3+\aaaa}\rrrr),
\end{align}
where
\begin{align*}
&c_0=\dddd{1}{2}(3\zzzz(-\aaaa)-\zzzz(-1-\aaaa)-2\zzzz(1-\aaaa)),\\
&c_1=-2\zzzz(-\aaaa)-\zzzz(-1-\aaaa),\\
&c_2=\dddd{1}{2}(\zzzz(-\aaaa)-\zzzz(-1-\aaaa)).
\end{align*}
\end{clm}
The numerical solution of equation \eqref{27_0} which uses approximation \eqref{30_1} for the fractional integral is computed with $u_0=u_1=0$ and
\begin{equation}\label{30_2}
u_n=\dddd{1}{\GGGG(\aaaa)+c_0h^\aaaa}\llll(\GGGG(\aaaa)F_n-h^\aaaa\llll( c_1 u_{n-1}+c_2 u_{n-2}+ \sum_{k=1}^{n-1} \dddd{u_{n-k}}{k^{1-\aaaa}}\rrrr)\rrrr).
\end{equation}
The numerical results for the maximum error and the order of numerical solution \eqref{30_2} for Equation 1 with $\aaaa=0.35,1.35$, Equation 2[3] with $\aaaa=0.5,1.5$ and Equation 3[4] with $\aaaa=0.65,1.65$ are given in Tables 8 and 9.
\setlength{\tabcolsep}{0.5em}
{ \renewcommand{\arraystretch}{1.1}
\begin{table}[ht]
	\caption{Maximum error and order of numerical solution \eqref{30_2} for $\aaaa=0.35$, $\aaaa=0.5$ and $\aaaa=0.65$.}
	\small
	\centering
  \begin{tabular}{ l | l  l | l  l | l  l }
		\hline
		\hline
		\multirow{2}*{ $\quad \boldsymbol{h}$}  & \multicolumn{2}{c|}{{\bf Equation\;1}} & \multicolumn{2}{c|}{{\bf Equation\;2[3]}}  & \multicolumn{2}{c}{{\bf Equation\;3[4]}} \\
		\cline{2-7}  
   & $Error$ & $Order$  & $Error$ & $Order$  & $Error$ & $Order$ \\ 
		\hline \hline
$0.025$    & $1.5\times 10^{-6}$  & $3.3224$  & $7.5\times 10^{-8}$   & $3.4558$    & $6.8\times 10^{-9}$   & $3.8511$       \\ 
$0.0125$   & $1.5\times 10^{-7}$  & $3.3369$  & $6.8\times 10^{-9}$   & $3.4784$    & $4.7\times 10^{-10}$  & $3.8689$       \\ 
$0.00625$  & $1.4\times 10^{-8}$  & $3.3437$  & $6.0\times 10^{-10}$  & $3.4894$    & $3.2\times 10^{-11}$  & $3.8802$       \\ 
$0.003125$ & $1.4\times 10^{-9}$  & $3.3469$  & $5.3\times 10^{-11}$  & $3.4947$    & $2.1\times 10^{-12}$  & $3.8875$       \\
\hline
  \end{tabular}
	\end{table}
	}
 \setlength{\tabcolsep}{0.5em}
{ \renewcommand{\arraystretch}{1.1}
\begin{table}[h]
	\caption{Maximum error and order of numerical solution \eqref{30_2} for $\aaaa=1.35$, $\aaaa=1.5$ and $\aaaa=1.65$.}
	\small
	\centering
\begin{tabular}{ l | l  l | l  l | l  l }
		\hline
		\hline
		\multirow{2}*{ $\quad \boldsymbol{h}$}  & \multicolumn{2}{c|}{{\bf Equation\;1}} & \multicolumn{2}{c|}{{\bf Equation\;2[3]}}  & \multicolumn{2}{c}{{\bf Equation\;3[4]}} \\
		\cline{2-7}  
   & $Error$ & $Order$  & $Error$ & $Order$  & $Error$ & $Order$ \\ 
		\hline \hline 
$0.025$    & $2.5\times 10^{-8}$  & $4.1054$  & $8.2\times 10^{-10}$   & $4.2036$   & $1.3\times 10^{-11}$ & $4.5161$       \\ 
$0.0125$   & $1.3\times 10^{-9}$  & $4.2189$  & $4.1\times 10^{-11}$   & $4.3337$   & $5.3\times 10^{-13}$ & $4.5832$       \\ 
$0.00625$   & $6.9\times 10^{-11}$  & $4.2743$  & $1.9\times 10^{-12}$   & $4.3980$  & $2.2\times 10^{-14}$ & $4.6179$        \\ 
$0.003125$ & $3.5\times 10^{-12}$  & $4.3047$  & $8.9\times 10^{-14}$   & $4.4325$  & $8.8\times 10^{-16}$ & $4.6233$        \\
\hline
  \end{tabular}
	\end{table}
	}
\subsection{ Numerical solution of order $\boldsymbol{4+\aaaa}$} 
By substituting the derivatives $y'_n,y''_n$ and $y'''_n$ in  \eqref{27_1} using the approximations
$$y'_n=\dddd{1}{h}\llll(\dddd{11}{6}y_n-3y_{n-1}+\dddd{3}{2}y_{n-2}-\dddd{1}{3}y_{n-3}\rrrr)+O\llll(h^3\rrrr),$$
$$y''_n=\dddd{1}{h^2}\llll(2 y_n-5 y_{n-1}+4 y_{n-2}-y_{n-3}\rrrr)+O\llll(h^2\rrrr),$$
$$y'''_n=\dddd{1}{h^3}\llll( y_n-3 y_{n-1}+3 y_{n-2}-y_{n-3}\rrrr)+O\llll(h\rrrr).$$
we obtain an approximation  for the fractional integral with order $4+\aaaa$.
\begin{clm} Suppose that $y(0)=y'(0)=y''(0)=y'''(0)=0$. Then
\begin{align}\label{32_1}
\dddd{h^\aaaa}{\GGGG(\aaaa)} \llll(c_0 y_n+c_1 y_{n-1}+c_2 y_{n-2}+c_3 y_{n-3}+\sum_{k=1}^{n-1}\dddd{y_{n-k}}{k^{1-\aaaa}}   \rrrr)=I^\aaaa y_n+O\llll(h^{4+\aaaa}\rrrr),
\end{align}
where
\begin{align*}
&c_0=\dddd{11}{6}\zzzz(-\aaaa)-\zzzz(-1-\aaaa)+\dddd{1}{6}\zzzz(-2-\aaaa)-\zzzz(1-\aaaa),\\
&c_1=-3\zzzz(-\aaaa)+\dddd{5}{2}\zzzz(-1-\aaaa)-\dddd{1}{2}\zzzz(-2-\aaaa),\\
&c_2=\dddd{3}{2}\zzzz(-\aaaa)-2\zzzz(-1-\aaaa)+\dddd{1}{2}\zzzz(-2-\aaaa),\\
&c_3=-\dddd{1}{3}\zzzz(-\aaaa)+\dddd{1}{2}\zzzz(-1-\aaaa)-\dddd{1}{6}\zzzz(-2-\aaaa).
\end{align*}
\end{clm}
The corresponding numerical solution of equation \eqref{27_0} is computed with $u_0=u_1=u_2=0$,  and
\begin{equation}\label{32_2}
u_n=\dddd{1}{\GGGG(\aaaa)+c_0h^\aaaa}\llll(\GGGG(\aaaa)F_n-h^\aaaa\llll( c_1 u_{n-1}+c_2 u_{n-2}+c_3 u_{n-3}+ \sum_{k=1}^{n-1} \dddd{u_{n-k}}{k^{1-\aaaa}}\rrrr)\rrrr).
\end{equation}
The numerical results for the maximum error and the order of numerical solution \eqref{32_2} for Equation 1 with $\aaaa=0.4$, Equation 2[4] with $\aaaa=0.5$ and Equation 3[9] with $\aaaa=0.6$ are given in Table 10.
\setlength{\tabcolsep}{0.5em}
{ \renewcommand{\arraystretch}{1.1}
\begin{table}[ht]
	\caption{Maximum error and order of numerical solution \eqref{32_2} for $\aaaa=0.4$, $\aaaa=0.5$ and $\aaaa=0.6$.}
	\small
	\centering
  \begin{tabular}{ l | l  l | l  l | l  l }
		\hline
		\hline
		\multirow{2}*{ $\quad \boldsymbol{h}$}  & \multicolumn{2}{c|}{{\bf Equation\;1}} & \multicolumn{2}{c|}{{\bf Equation\;2[4]}}  & \multicolumn{2}{c}{{\bf Equation\;3[9]}} \\
		\cline{2-7}  
   & $Error$ & $Order$  & $Error$ & $Order$  & $Error$ & $Order$ \\ 
		\hline \hline
$0.025$    & $1.7\times 10^{-9}$   & $4.3144$  & $1.3\times 10^{-9}$    & $4.4072$    & $1.9\times 10^{-11}$   & $4.5234$       \\ 
$0.0125$   & $8.1\times 10^{-11}$  & $4.3618$  & $5.8\times 10^{-11}$   & $4.4596$    & $7.8\times 10^{-13}$   & $4.5641$     \\ 
$0.00625$  & $3.9\times 10^{-12}$  & $4.3819$  & $2.6\times 10^{-12}$   & $4.4813$   & $3.3\times 10^{-14}$    & $4.5884$        \\ 
$0.003125$ & $1.9\times 10^{-13}$  & $4.3873$  & $1.2\times 10^{-13}$   & $4.4895$    & $1.4\times 10^{-15}$   & $4.5413$        \\
\hline
  \end{tabular}
	\end{table}
	}


\begin{thebibliography}{99}
\bibitem{AbramowitzStegun1964}M. Abramowitz, I. A. Stegun: {\it Handbook of Mathematical Functions with Formulas, Graphs, and Mathematical Tables,} Dover, New York, 1964.
\bibitem{AnjaraSolofoniaina2014} F. Anjara and J. Solofoniaina: {\it Solution of general fractional oscillation relaxation equation by Adomian’s method,} Gen. Math. Notes, {\bf 20(2)} (2014), 1--11.
\bibitem{Atkinson1997} K. E. Atkinson: {\it The Numerical Solution of Integral Equations of the Second Kind,} Cambridge  University Press, (1997).
\bibitem{BrayanovZlateva2005} I. Brayanov, K. Zlateva: {\it Numerical Methods},  Rousse University (2005).
\bibitem{Cartea2007}  A. Cartea, D. del Castillo-Negrete:  {\it Fractional diffusion models of option prices in markets with jumps,}  Physica A, {\bf 374(2)} (2007), 749--763.
\bibitem{Diethelm2010} K. Diethelm:  {\it The analysis of fractional differential equations: an application-oriented exposition using differential operators of Caputo type,} Springer; 2010.
\bibitem{DiethelmFord2012} K. Diethelm, N. J. Ford: {\it Volterra integral equations and fractional calculus: Do neighboring solutions intersect?,} Journal of Integral Equations Applications, {\bf 24(1)} (2012), 25--37.
\bibitem{Dimitrov2014} Y. Dimitrov: {\it Numerical approximations for fractional differential equations,} Journal of Fractional Calculus and Applications, {\bf 5(3S)} (2014), No. 22, 1--45. 
\bibitem{Dimitrov2015_1} Y. Dimitrov: {\it A second order approximation for the Caputo fractional derivative,} {\bf arXiv:1502.00719} (2015).
\bibitem{Dimitrov2015_2} Y. Dimitrov: {\it Three-Point Compact Approximation for the Caputo Fractional Derivative,} {\bf 	arXiv:1510.01619} (2015).
\bibitem{DingLi2016} H. Ding, C. Li:  {\it High-order algorithms for Riesz derivative and their applications (III)}, Fract. Calc. Appl. Anal., 19(1), (2016), 19--55.
\bibitem{EslahchiDehghanParvizi2014} M.R. Eslahchi, M. Dehghan and M. Parvizi: {\it Application of the collocation method for solving nonlinear fractional integro-differential equations,} Journal of Computational and Applied Mathematics, {\bf 257} (2014), 105--128.
\bibitem{GorenfloMainardi2008} Rudolf Gorenflo, Francesco Mainardi: {\it FRACTIONAL CALCULUS: Integral and Differential Equations of Fractional Order}, CISM Lecture notes (2008).
\bibitem{GulsuOzturkAnapali2013} M. G\"ulsu, Y. \"Ozt\"urk,  A. Anapal{\i}: {\it Numerical approach for solving fractional relaxation-oscillation equation,} Applied Mathematical Modelling, {\bf 37(8)} (2013), 5927--5937.
\bibitem{HuangLiZhaoDuan2011} L. Huang, X. F. Li, Y. L. Zhao and X. Y. Duan: {\it Approximate solution of fractional integro-differential equations by Taylor expansion method,} Comput. Math. Appl., {\bf 62} (2011), 1127--1134.
\bibitem{Ishteva2005} M. Ishteva: {\it Properties and applications of the Caputo fractional operator,} Master thesis,  University of Karlsruhe  (2005).
\bibitem{CheungKac2002} V. Kac, P. Cheung: {\it Quantum Calculus,} Universitext, Springer, (2002).
\bibitem{Kouba2013} O. Kouba: {\it Bernoulli polynomials and applications,}  Lecture Notes (2013).
\bibitem{Lampert2001} V. Lampret: {\it The Euler–Maclaurin and Taylor Formulas:
twin, elementary derivations,} Mathematics magazine, {\bf 74(2)} (2001), 109--122.
\bibitem{LiWuDing2014} C. Li, R. Wu, H. Ding: {\it High-order approximation to Caputo derivatives and Caputo-type advection-diffusion equations,} Communications in Applied and Industrial Mathematics, {\bf 6(2)} (2014), e-536.
\bibitem{Mainardi1996} F, Mainardi: {\it Fractional relaxation-oscillation and fractional diffusion-wave phenomena,} Chaos, Solitons $\&$ Fractals, {\bf 7(9)} (1996), 1461--1477.
\bibitem{MainardiGorenflo2007} F. Mainardi, R. Gorenflo: {\it Time-fractional derivatives in relaxation processes: a tutorial survey}, Fractional Calculus and Applied Analysis, {\bf 10(3)}, (2007), 269--308.
\bibitem{Mokhtary2016} P. Mokhtary: {\it Discrete Galerkin method for fractional integro-differential equations,} Acta Mathematica Scientia, 
 {\bf 36(2)} (2016), 560--578.
\bibitem{PerezYamilet2008} D. P\'erez, Q. Yamilet: {\it A survey on the Weierstrass approximation theorem,}, Divulgaciones Matem\'aticas, {\bf 16(1)}, (2008), 231--247.
\bibitem{Podlubny1999} I. Podlubny,  Fractional Differential Equations. Academic Press, San Diego; 1999.
\bibitem{Sidi2004} A. Sidi: {\it Euler-Maclaurin expansions for integrals with endpoint singularities: A new perspective,}  Numer. Math. {\bf 98(2)} (2004), 371--387.
\bibitem{TadjeranMeerschaertScheffer2006} C. Tadjeran, M. M. Meerschaert and H. P. Scheffer: {\it A second-order accurate numerical approximation for the fractional diffusion equation,} Journal of Computational Physics, {\bf 213} (2006), 205--213.
\bibitem{TianZhouDeng2012} W.Y. Tian,  H. Zhou and  W. Deng: {\it A class of second order difference approximation for solving space fractional diffusion equations,} Mathematics of Computation, {\bf 84(294)} (2012), 1703--1727.
\bibitem{WuDingLi2014} R. Wu, H. Ding and C. Li: {\it Determination of coefficients of high-order schemes for Riemann-Liouville derivative,} The Scientific World Journal (2014), Article ID 402373.
\bibitem{YanPalFord2014} Y. Yan, K. Pal and N. J. Ford: {\it Higher order numerical methods for solving fractional differential equations,}
BIT Numerical Mathematics, {\bf 54(2)}  (2014), 555--584.
\end{thebibliography}
\end{document}